\documentclass[reqno,oneside, 11pt,a4paper]{amsart}

\usepackage{geometry}                		
\usepackage{graphicx}				
	
\usepackage{amssymb}

\usepackage[utf8]{inputenc}
\usepackage[english]{babel}

\usepackage{amsthm}

\usepackage{graphicx}

\usepackage{amsmath}
\usepackage{yfonts}

\usepackage[usenames,dvipsnames]{xcolor}
\usepackage[colorlinks=true,linkcolor=Blue,citecolor=teal]{hyperref}
\usepackage{amsmath,amssymb,epsf}
\usepackage{amsfonts,amsbsy,amscd,stmaryrd}
\usepackage{latexsym,amsopn}
\usepackage{amsthm,mathtools}
\usepackage{mathrsfs}
\usepackage{slashed}

\usepackage{mathrsfs}
\usepackage{esint}
\usepackage[mathscr]{eucal}

\allowdisplaybreaks

\usepackage{bbm}

\usepackage{graphicx}
\usepackage{autonum} 

\usepackage{color}
\definecolor{Red}{rgb}{1.,0.,0.}
\newcounter{smallarabics}
\newenvironment{arabicenumerate}
{\begin{list}{{\normalfont\textrm{(\arabic{smallarabics})}}}
  {\usecounter{smallarabics}\setlength{\itemindent}{0cm}
   \setlength{\leftmargin}{5ex}\setlength{\labelwidth}{4ex}
   \setlength{\topsep}{0.75\parsep}\setlength{\partopsep}{0ex}
   \setlength{\itemsep}{0ex}}}
{\end{list}}

\newcounter{smallroman}

\newcommand{\ben}{\begin{arabicenumerate}}
\newcommand{\een}{\end{arabicenumerate}}

\newtheorem{theorem}{Theorem}[section]

\newtheorem{proposition}[theorem]{Proposition}
\newtheorem{lemma}[theorem]{Lemma}

\theoremstyle{definition}
\newtheorem{definition}[theorem]{Definition}
\newtheorem{remark}[theorem]{Remark}
\newtheorem{example}[theorem]{Example}
\newcommand{\beq}{\begin{equation}}
\newcommand{\eeq}{\end{equation}}

\newcommand{\bea}{\begin{aligned}}
\newcommand{\eea}{\end{aligned}}

\newcommand{\bex}{\begin{example}}
\newcommand{\eex}{\end{example}}
\def\bel{\begin{lemma}}
\def\eel{\end{lemma}}
\def\bet{\begin{theorem}}
\def\eet{\end{theorem}}
\def\bed{\begin{definition}}
\def\eed{\end{definition}}
\def\ber{\begin{remark}}
\def\eer{\end{remark}}

\renewcommand{\leq}{\leqslant}
\renewcommand{\geq}{\geqslant}

\def\rr{{\mathbb R}}

\def\cc{{\mathbb C}}

\def\sss{{\mathbb S}}

\def\part{{\rm par}}

\def\cinf{C^\infty}

\usepackage{calrsfs}
\DeclareMathAlphabet{\pazocal}{OMS}{zplm}{m}{n}

\def\cR{{\pazocal R}}
\def\cD{{\pazocal D}}
\def\cU{{\pazocal U}}

\def\wf{{\rm WF}}

\def\loc{{\rm loc}}

\let\Im\relax
\let\Re\relax

\DeclareMathOperator{\Im}{Im}
\DeclareMathOperator{\Re}{Re}

\DeclareMathOperator{\WF}{WF'}

\newcommand{\qeds}{\qed\medskip}

\def \p{ \partial}

\def\12{\frac{1}{2}}
\def\14{\frac{1}{4}}

\DeclareMathOperator{\supp}{supp}

\newcommand{\one}{\boldsymbol{1}}

\def\c{{\rm c}}

\def\12{\frac{1}{2}}

\def\Diff{{\rm Diff}}

\def\bep{\begin{proposition}}
\def\eep{\end{proposition}}

\newcommand{\bra}{\langle }
\newcommand{\ket}{\rangle }

\DeclareSymbolFont{boldoperators}{OT1}{cmr}{bx}{n}
\SetSymbolFont{boldoperators}{bold}{OT1}{cmr}{bx}{n}

\makeatletter

\newcommand*{\defeq}{:=}				
							
\makeatother

\def\Op{{\rm Op}}
\def\WF{{\rm WF}}

\def\fantom{\\ &\phantom{=}\,}
\def\cf{C^\infty}
\def\zero{{\rm\textit{o}}}

\def\c{{\rm c}}

\def\varm{{\rm\textit{m}}}

\def\diag{{\rm diag}}

\def\loc{{\rm loc}}

\let\origmaketitle\maketitle
\def\maketitle{
  \begingroup
  \def\uppercasenonmath##1{} 
  \let\MakeUppercase\relax 
	\origmaketitle
  \endgroup
}

\def\beproof{
\noindent{\bf Proof.}\ \ }
\renewenvironment{proof}{\beproof}{\qeds}
\newenvironment{refproof}[1]{\smallskip
\noindent{\bf Proof of {#1}.}\ \ }{\qeds}

\def\mm{{\scriptscriptstyle(-)}}

\def\st{{ \ |\  }}
\newcommand{\norm}[1]{\left\|{#1}\right\|}
\newcommand{\module}[1]{\left|#1\right|}
\newcommand{\braket}[1]{{\langle {#1}\rangle }}

\makeatletter
\newsavebox\myboxA
\newsavebox\myboxB
\newlength\mylenA

\newcommand*\xoverline[2][0.75]{%
    \sbox{\myboxA}{$\m@th#2$}%
    \setbox\myboxB\null
    \ht\myboxB=\ht\myboxA%
    \dp\myboxB=\dp\myboxA%
    \wd\myboxB=#1\wd\myboxA
    \sbox\myboxB{$\m@th\overline{\copy\myboxB}$}
    \setlength\mylenA{\the\wd\myboxA}
    \addtolength\mylenA{-\the\wd\myboxB}%
    \ifdim\wd\myboxB\langle \wd\myboxA%
       \rlap{\hskip 0.8\mylenA\usebox\myboxB}{\usebox\myboxA}%
    \else
        \hskip -0.9\mylenA\rlap{\usebox\myboxA}{\hskip 0.9\mylenA\usebox\myboxB}%
    \fi}
\makeatother

\makeatletter
\newsavebox\xmyboxA
\newsavebox\xmyboxB
\newlength\xmylenA

\newcommand*\xxoverline[2][0.75]{%
    \sbox{\xmyboxA}{$\scriptstyle\m@th#2$}%
    \setbox\xmyboxB\null
    \ht\xmyboxB=\ht\xmyboxA%
    \dp\xmyboxB=\dp\xmyboxA%
    \wd\xmyboxB=#1\wd\xmyboxA
    \sbox\xmyboxB{$\scriptstyle\m@th\overline{\copy\xmyboxB}$}
    \setlength\xmylenA{\the\wd\xmyboxA}
    \addtolength\xmylenA{-\the\wd\xmyboxB}%
    \ifdim\wd\xmyboxB\langle \wd\xmyboxA%
       \rlap{\hskip 0.8\xmylenA\usebox\xmyboxB}{\usebox\xmyboxA}%
    \else
        \hskip -0.9\xmylenA\rlap{\usebox\xmyboxA}{\hskip 0.9\xmylenA\usebox\xmyboxB}%
    \fi}
\makeatother

\newcommand{\wfl}[2][s]{{\rm WF}'^{\,(#1)}_{{\ifempty{#2} h(z) \else \ifnum #2=0 {} \else  \ifnum #2=1 \Zweight^{-\12} \else  \ifnum #2=2 \module{\Im z}^{-\12} \else  \ifnum #2=3 \bra{\Im z}\ket^{-\12}  \else  \ifnum #2=4 h_1(z)  \else  \ifnum #2=5 h_2(z)  \else  \ifnum #2=6 h_1h_2(z)  \else  \ifnum #2=7  \bra z \ket^{-\infty} \else  \ifnum #2=12  \bra z \ket^{-\12} \fi\fi\fi\fi\fi\fi\fi\fi\fi\fi}}}

\def \ifempty#1{\def\temp{#1} \ifx\temp\empty }
\newcommand{\Oreg}[2][{+s}]{\pazocal{O}_{H^{*}\to H^{* #1}}(\ifempty{#2} h(z) \else \ifnum #2=0 1\fi \ifnum #2=1 \Zweight^{-1}\fi \ifnum #2=12 \Zweight^{-\12}\fi\fi)}
\newcommand{\Oregsh}[1][{+s}]{\pazocal{O}_{H^*\to H^{* #1}}}
\newcommand{\Onorm}[1]{\pazocal{O}_{\cf\to\cf}(\ifempty{#1} h(z) \else \ifnum #1=0 1\fi \ifnum #1=1 \Zweight^{-\12}\fi\ifnum #1=12 \Zweight^{-\12}\fi\fi)}
\newcommand{\Onormsh}{\pazocal{O}_{\cf\to\cf}}
\newcommand{\Olambda}[1]{\pazocal{O}_{\cD'_\Lambda}(\ifempty{#1} h(z) \else \ifnum #1=0 1\fi \ifnum #1=1 \Zweight^{-1}\fi\fi)}
\newcommand{\Osmooth}[1]{\pazocal{O}_{C^\infty}(\ifempty{#1} h(z) \else \ifnum #1=0 1\fi \ifnum #1=1 \Zweight^{-1}\fi\fi)}

\renewcommand{\geq}{\geqslant}
\renewcommand{\leq}{\leqslant}

\def\Diff{{\rm Diff}}

\usepackage{pdfpages}

\title[Feynman spectral action of the wave operator on asymptotically de Sitter spaces]{\Large Feynman spectral action of the wave operator on asymptotically de Sitter spaces }

\begin{document}
	
	
	%
	\author[]{\large Ruben \textsc{Zeitoun}}
	
	\begin{abstract} 
		In this paper, we investigate the wave operator $\square_g$ on non-trapping (at all energies) even asymptotically de Sitter spaces. We construct a Feynman operator on the conformal extension of asymptotically de Sitter spaces and give
		a proof of uniform microlocal estimates for the Feynman operator in this setting. This enables the study of the
		Lorentzian "spectral" zeta functions in asymptotically de Sitter and the construction of a "spectral" action of the Feynman propagator.
	\end{abstract}

	{\Large\maketitle}
	
%

\author{Ruben \textsc{Zeitoun}}

%



\section{Introduction }

\subsection{Motivations}
 The deep links between geometry and the spectral theory of the Laplace--Beltrami operator (the natural operator associated to a metric) have inspired important developments in physics at the interface of General Relativity and Quantum Field Theory (QFT). One of the most notable  is non-commutative geometry, which can be used to derive the standard model of particles \cite{Con}. However, one fundamental obstacle remains. Indeed, the physical interpretation requires the manifolds $(M,g)$ to be Lorentzian (i.e., to have a dimension of time and the rest of space) instead of Riemannian (i.e.~to only have spatial dimensions). The tremendous price is that the Laplace--Beltrami operator becomes the wave operator, which does not have the properties essential to prove most theorems (ellipticity and positivity). It is therefore outside of the scope of most of the developments mentioned previously. Despite this fact, several authors recently showed that a spectral theory can be defined and similar results can be obtained if we have symmetries or a good asymptotic behavior.  \smallskip

One of those  relations can be explained as follows. Let 
$$
\Delta_g= -|g|^{-\12}\sum_{i,j=1}^n \frac{\p}{\p x_i} |g|^{\12} g^{ij} \frac{\p}{\p x_j},
$$
 be the Laplace-Beltrami operator on a \emph{compact Riemannian manifold} $(M,g)$. 
It is well known that $\Delta_g$ has a discrete spectrum in the compact case. Therefore we can introduce the \emph{spectral zeta function}
$$
\zeta_{\Delta_g}(\alpha)=\sum_{\lambda\in \sigma_p(\Delta_g)\setminus \{0\}} \lambda^{-\alpha}.
$$
The properties of $\zeta_{\Delta_g}$ as a function of the complex variable $\alpha$ are  summarized by   classic results from elliptic theory:

\emph{ {\rm(Minakshisundaram--Pleijel \cite{MinPle}, Seeley  \cite{See}, Connes, Kalau--Walze  \cite{KalWal}, Kastler  \cite{Kas})}
The trace density $\alpha\mapsto \Delta^{-\alpha}(x,x)$ (integrand of the trace) is holomorphic on $\Re \alpha>\frac{n}{2}$ and has a meromorphic extension in $\alpha\in \cc$ with poles in $\{\frac{n}{2},\frac{n}{2}-1,\dots ,1  \}$, valued in $C^\infty$ functions in regard to the variable $x\in M$. (Here, for $x\in M$, $\Delta_g^{-\alpha}(x,x)$ is the restriction to the point $(x,x)$ of the diagonal $x=y$ of the Schwartz kernel $\Delta_g^{-\alpha}(x,y)$ associated to $\Delta_g^{-\alpha}$.)  Furthermore, if $\dim(M)=n\geq 4$, then
\begin{equation}\label{eq1}
{\rm res}_{\alpha=\frac{n}{2}-1} \Delta_g^{-\alpha}(x,x)=\frac{{R_g(x)}}{(12)2^n\pi^{\frac{n}{2}} \Gamma(\frac{n}{2}-1) }.
\end{equation}
}

These type of results have inspired important progress in relativistic physics at {the interface of relativity and QFT} \cite{Haw,Con,ChaCon,ConMar,Sui}. One of the principal reasons for this is that the identity \eqref{eq1} shows that General Relativity can be derived from a \emph{spectral action} using a \emph{principle of least action}: indeed the equation $\delta_gR_g=0$ is equivalent to \emph{Einstein's equations} for the metric $g$, where the term in the left hand side is given by the spectral theory of $\Delta_g$.

However, it is necessary to underline that \emph{the hypothesis that $(M,g)$ is Riemannian and compact is far from the physical reality which requires $(M,g)$ to be of Lorentzian signature $(+,-,...,-)$}. This leads to difficulties which have constituted an insurmountable obstacle for a long time. Even though it is possible, in some measure, to formally change the signature to formulate a large number of conjectures, their mathematical proof encounters very serious problems. First of all, the Laplace-Beltrami operator $\square_g$, also known as the \emph{wave operator} in this case, is not elliptic, which is a key prerequisite of techniques used in proofs. Moreover, $\square_g$ is not bounded from below and there is no obvious reason why it could be interpreted as a selfadjoint operator. \emph{Without (essential) selfadjointness most of the statements using spectral theory cannot even be defined}.

\subsection{State of the art}
The first approach was to study highly symmetric spacetimes \cite{DerSie}. However, in the last few years Vasy \cite{VasySelf}, followed by Nakamura-Taira \cite{NakTai}, showed that for different type of  perturbations of Minkowski spacetime the essential self-adjointness of the wave operator remains true.  

{The spectral theory of geometric differential operators on a Lorentzian manifold $(M,g)$ is an emerging topic with surprising features. In spite of   non-ellipticity, the Laplace--Beltrami or wave operator $\square_g$ has been shown to be essentially self-adjoint in a variety of settings, including static spacetimes \cite{DerSie}, de Sitter \cite{DerGas} and anti-de Sitter \cite{DelGuiMon}  asymptotically Minkowski spacetimes  \cite{VasySelf,NakTai,Nakamura2022,JMS}, and Cauchy-compact asymptotically static spacetimes \cite{Nakamura2022a}, see also \cite{Tadano2019,taira,kaminski,Taira2020a,cdv,Taira2022,DerGas} for other results on spectral properties of $\square_g$ and the limiting absorption principle and \cite{GHV,Vasy2017b,GWfeynman,Gerard2019b,vasywrochna,MolodhykVasy,DerGas} for the closely related subject of Feynman propagators.}


However, even though it is possible to talk about a spectral theory despite the  non-ellipticity of $\square_g$, it is far from obvious that it is linked to geometric quantities like the scalar curvature whatsoever. The proofs known in the Riemannian case use either the heat kernel or the elliptic pseudo-differential calculus. It is difficult to imagine a generalization in the non-elliptic case.
In spite of this, Dang--Wrochna recently proved the following result with microlocal techniques.

\emph{ (Dang--Wrochna  \cite{DanWro})
 $(M,g)$ is a non-trapping Lorentzian scattering space of even dimension $n\geq 4$. For all $\epsilon>0$, the Schwartz kernel of $(\square_g-i\epsilon)^{-\alpha}$ has for $\Re \alpha>\frac{n}{2}$ a well-defined on-diagonal restriction $(\square_g-i\epsilon)^{-\alpha}(x,x)$ which extends as a meromorphic function of $\alpha\in \mathbb C$ with poles at $\{\frac{n}{2},\frac{n}{2}-1,...,1\}$. Furthermore,\begin{equation}\label{eq2}
\lim_{\epsilon\to 0^+}{\rm res}_{\alpha=\frac{n}{2}-1}(\square_g-i\epsilon)^{-\alpha}(x,x)=\frac{R_g(x)}{i6(4\pi)^{\frac{n}{2}}\Gamma(\frac{n}{2}-1)},\end{equation}
where $R_g(x)$ is the scalar curvature at $x\in M$.
}

This means in particular that \emph{gravity is indeed described by a spectral action in physical Lorentzian signature}, at least for asymptotically Minkowski spaces. See also \cite{DanVasWro} for the spectral action of the Dirac operator.

This leads to the question: \emph{does it still hold for non-static spacetimes with non-static asymptotic behavior?}
The \emph{(global) de Sitter spacetime} (see Section \ref{deS}) is a great case to study because it is extremely non-static while having symmetries and an interesting behavior at infinity. It is also used to model the expansion of the universe, allowing us to test  ideas in a \emph{cosmological setting}.

\subsection{Main result}
In the same way that asymptotically Minkowski spacetimes are the Lorentzian version of asymptotically Euclidean spaces, \emph{asymptotically de Sitter spaces} (or more exactly de Sitter-like spaces) are the Lorentzian version of \emph{asymptotically hyperbolic spaces}. 



Using methods first introduced by Vasy \cite{VasyEst} and further developed  by Dyatlov and Zworski \cite{DyaZwo}  we were able to construct the four distinguished propagators (advanced, retarded, Feynman and Anti-Feynman) which are the core of the proof in the Minkowski case. In contrast to the Minkowski space setting those operators turned out to have a limited use. The reason is that in Minkowski space the "decay" asymptotics and the "wavefront" asymptotics coincide and therefore selecting the $L^2$ propagator gives the Feynman propagator which then must be equal to the resolvent. This turned out to be different in the de Sitter case.

Without this fact there was very little we could say from this approach about the $L^2$ case. So much so that even the question of selfadjointness is elusive. We knew from \cite{Rum} and \cite{DerGas} that in the exact de Sitter case that the wave operator was selfadjoint. The question of selfadjointness is however a work in progress.

We are however not able to generalize previous result to this setting given that the resolvent is not the Feynman operator constructed here and might not even have a Feynman wavefront set which is the core of the proof of the theorem. We are however able to construct a "spectral action" for the Feynman operator:

\begin{theorem}\label{mainthm2intro}
	Let $(M,g)$ be a non-trapping even asymptotically de Sitter  space of even dimension $n$. Then 
	the Schwartz kernel $K_{\alpha,k}(.,.) $ of $ R_F^{(\alpha,k)}(\pm i\epsilon) $ exists as a family of distributions near the diagonal depending holomorphically in $\alpha$ on the half-plane $\Re\alpha>2-k$.
	Its restriction on the diagonal
	$K_{\alpha,k}(x,x)$ exists and is holomorphic for $\Re\alpha>\frac{n}{2} $, and it extends as a meromorphic 
	function of $\alpha$ with simple poles along the arithmetic progression
	$\{\frac{n}{2}, \frac{n}{2}-1,\dots,1\}$. Furthermore,
	$$
	\lim_{\epsilon\rightarrow 0^+} {\rm res}_{\alpha=\frac{n}{2}-1}R_F^{(\alpha,k)}(\pm i\epsilon)(x,x)= 
	\pm\frac{ i \,R_g(x)}{6(4\pi)^{\frac{n}{2}}(\frac{n}{2}-2)!},
	$$
	where $
	R_{ F}^{(\alpha,k)}(\mu+i\epsilon,.)=\frac{1}{2\pi i}P_0(\lambda(\mu+i\epsilon))^k\int_{\gamma_\epsilon}(z-i\epsilon)^{-\alpha-k}	R_F(\lambda(z+\mu))dz
	$ are the "complex powers" of the Feynman propagator $R_{ F}$.
\end{theorem}

This shows that even if the spectral action, i.e equation \eqref{eq2}, might not be generalizable to any well behaved Lorentzian manifold, this version on the Feynman operator might. Strictly speaking, the left hand side is not spectral, it is however globally defined and one can argue it is canonical.
\subsection{Structure of the proofs}

The proof of Theorem \ref{mainthm2intro} starts with constructing the Feynman propagator obtain using microlocal estimates and Fredholm theory.  We then need to check that the Feynman propagator has a uniform Feynman wavefront set and to modify the setting use in \cite{DanWro} to accommodate the change from the resolvent to the Feynman propagator.

\subsubsection{Content}

The paper is organized as follows.

In sections \ref{Mic} and \ref{Sem}, we introduce preliminaries on microlocal analysis and semi-classical analysis that will be needed in the following.

In section \ref{deS}, we introduce the geometrical setting of first the exact de Sitter space, then, in section \ref{Asy} asymptotically de Sitter space. We will also familiarize ourselves with the wave operator on those spaces.

In section \ref{DisInv}, largely following Vasy \cite{VasyEst} and Dyatlov--Zworski \cite{DyaZwo} we give an overview of the construction of the four microlocally distinguished inverses: advanced, retarded, Feynman and anti-Feynman. For this we first need to understand the dynamical properties of the geodesics in order to construct microlocal estimates as well as providing some of the theory of uniform operatorial wavefront sets.

In section \ref{Had}, we go in depths into the Hadamard parametrix and how it relates to the spectral action. In section \ref{Fey} we link one of those propagator constructed in section \ref{DisInv} to the Hadamard parametrix to obtain the spectral action.

In the appendix, we discuss some classic results of Fredholm theory and useful lemmas in microlocal analysis.

\section{Microlocal analysis for variable order}\label{Mic}

Our goal is to define pseudo-differential calculus with variable order in a way that is analogous to the constant order case. Therefore we will not prove those results which are analogous, instead we will refer to \cite{Hintzmicroloc} for the microlocal setting and \cite{DyaZwo} Appendix E for the semi-classical one. This section should be read in parallel to the aforementioned resources.

To construct a framework sufficient for the complete proof Theorem \ref{invest} we would need to define semi-classical  pseudo-differential calculus with variable order. We refer to \cite{DyaZwo} for the construction of semi-classical pseudo-differential calculus with constant order.

\subsection{Pseudo-differential calculus with variable order}

We will now introduce semi-classical operators with variable order.

\begin{definition}
	Let $\rm m\in S^0(\rr^n;\rr^N)$, $n,N\in \mathbb N$ and $\delta>0$. Then the space of symbols of order $\rm m$ 
	\beq
	S^{\rm m}_\delta(\rr^n;\rr^N)\subset C^\infty(\rr^n\times \rr^N, (0,h_0))
	\eeq
	where $h_0$ is a fixed constant, consists of all functions $a=a(x,\xi,h)$ which for all $\alpha\in \mathbb R^n_0, \beta\in \mathbb R^N_0$ satisfy the estimate
	\beq
	\vert \p^\alpha_x\p^\beta_\xi a(x,\xi,h)\vert\leq C_{\alpha,\beta}\langle \xi\rangle^{\rm m-\vert \beta\vert+ \delta(\vert\alpha\vert+\vert\beta\vert)}.
	\eeq
	for some constants $C_{\alpha,\beta}$. We equip this space with the family of semi-norms
	\beq
	\norm{a}_{{\rm m},k}:= \sup_{h\in (0,h_0)}\sup_{(x,\xi)\in \rr^n\times\rr^N}\max_{\vert \alpha\vert+\vert\beta\vert\leq k}\langle\xi\rangle^{-{\rm m}+\vert \beta\vert-\delta(\vert\alpha\vert+\vert\beta\vert)}\vert \p^\alpha_x\p^\beta_\xi a(x,\xi)\vert
	\eeq
	for $k\in\rr$.
\end{definition}
With this definition we can see that $b=\langle \xi\rangle^{\rm m(x,\xi)}\in S^{\rm m}_\delta(\rr^n;\rr^N)$ since the terms which contains derivatives of $\rm m$ are of the form $\partial^\alpha_x\partial^\beta_\xi {\rm m} \ \p^\gamma_\xi \log(\langle\xi\rangle)\langle\xi\rangle^{\rm m-\sigma}$. $\rm m$ and the derivatives of $\log (\langle\xi\rangle)$ are in $S^0(\rr^n;\rr^N)$ and $\log (\langle\xi\rangle)$ grows slower than $\langle\xi\rangle^{\delta(\vert\alpha\vert+\vert \beta\vert)}$ so it satisfies the symbol estimate.

We also have, analogously to the constant order case, that the multiplication satisfies
\beq
S^{\rm m}_\delta(\rr^n;\rr^N)\times S_\delta^{\rm m'}(\rr^n;\rr^N)\to S_\delta^{\rm m+\rm m'}(\rr^n;\rr^N)
\eeq
and the derivation satisfies 
\beq
\p^\alpha_x\p^\beta_\xi: S^{\rm m}_\delta(\rr^n;\rr^N)\to S^{\rm m-\vert\beta\vert}_\delta(\rr^n;\rr^N)
\eeq
\begin{definition}
	Let ${\rm m}\in S^{0}(\rr^n;\rr^N)$. A symbol $a\in S^{\rm m}_\delta(\rr^n;\rr^N)$ is said to be elliptic if there exists a symbol $b\in S^{-\rm m}_\delta(\rr^n;\rr^N)$ such that $ab-1\in S^{-1}_\delta(\rr^n;\rr^N)$.
\end{definition}

We can now define the quantization of symbols of variable order by
\beq
(\Op(a)u)(x)=(2\pi)^{-n}\int_{\rr^n}\int_{\rr^N}e^{i(x-y).\xi}a(x,\xi)u(y)dyd\xi, \ \ u\in\mathcal S(\rr^n)
\eeq
where $\mathcal S(\rr^n)$ is the space of Schwartz functions.

The space of pseudo-differential operator of variable order $\rm m$ is $\Psi_\delta^{\rm m}(\rr^n):=\Op(S_\delta^{\rm m}(\rr^n;\rr^n))$ and it can be shown that the two spaces are isomorphic as Fréchet spaces with the induced topology. 

One could have also defined the quantization as the right quantization: 
\beq
(\Op_R(a)u)(x)=(2\pi)^{-n}\int_{\rr^n}\int_{\rr^N}e^{i(x-y).\xi}a(y,\xi)u(y)dyd\xi, \ \ u\in\mathcal S(\rr^n)
\eeq
and also this gives a different operator in the general case the space of pseudo-differential operators defined this way is the same.

\begin{definition}
	We define the symbol $\sigma(A)$ of a pseudo-differential operator $A\in \Psi^{\rm m}(\rr^n)$ as the unique symbol such that $\Op(\sigma(A))=A$.
	
	The principal symbol $\sigma_{\rm m}(A)$ of $A\in \Psi_\delta^{\rm m}(\rr^n)$ is the equivalent class 
	
	\beq
	\sigma_{\rm m}(A):=[\sigma (A)]\in S_\delta^{\rm m}(\rr^n;\rr^n)/ S_\delta^{\rm m-1}(\rr^n;\rr^n)
	\eeq
\end{definition}

\begin{proposition}
	 $\sigma_{\rm m}({\rm Op}(a))=[a]$.

\end{proposition}

\begin{remark}
	Here an asymptotic sum does not refer to a convergent sum. It is rather a way to specify an asymptotic behavior but we do not care about the low frequency behavior therefore there is an infinite number of ways to realize a symbol that differs from the symbol of $\sigma(A\circ B)$ by an element of $S^{-\infty}(\rr^n)$. We leave the details to \cite{hintz20}. However the following Lemma will highlight why such an error on the symbol is not an issue.
	\end{remark}
	
\begin{proposition}
	A residual operator $A\in \Psi^{-\infty}(\rr^n)$ is continuous as a map
	\beq
	A:\mathcal S'(\rr^n)\to \mathcal S'(\rr^n)\cap C^\infty(\rr^n).
	\eeq
\end{proposition}

\begin{definition}	
	We say that a pseudo-differential operator $A$ is elliptic if its principal symbol $\sigma_m(A)$ is elliptic.
\end{definition}

\subsection{Sobolev spaces}
\begin{lemma}
	Let $A\in \Psi_\delta^{0}(\rr^n)$, then $A$ is a bounded operator on $L^2(\rr)$.
\end{lemma}

We can now define Sobolev spaces of variable order.

\begin{definition}\label{varsob}
	Let $A\in \Psi_\delta^{\rm m}(\rr^n)$ be an elliptic operator. We define the Sobolev spaces of variable order $\rm m$ as the subset of $L^2(\rr^n)$ such that 
	\beq
	\norm{u}_{H^{\rm m}(\rr^n)}:=\norm{Au}_{L^2(\rr^n)}<+\infty
	\eeq
	But to fix a definition we take $A:=\langle D\rangle^m$. Here we define $\langle D\rangle^{\rm s}$ as ${\rm Op}(\langle \xi\rangle^{-\rm s})$.
\end{definition}

For simplicity we will omit the $\delta$ in the rest of this section.

\begin{proposition}
	Let ${\rm s,m}\in S^0(\rr^n,\rr^n)$ and $A\in \Psi^{m}(\rr^n)$, then $A:H^{\rm s}(\rr^n)\to H^{\rm s-\rm m}(\rr^n)$ is bounded.
	
	The definition of the Sobolev spaces does not depend on $A$ and the norms are equivalent for any $A$.
\end{proposition}

	We also define $H_{\rm loc}^{\rm s}(\rr^n)$ as distributions that are locally in $H^{\rm s}(\rr^n)$ and $H_{\rm comp}^{\rm s}(\rr^n)$ as elements of $H_{\rm loc}^{\rm s}(\rr^n)$ with compact support.

We will omit the construction of pseudo-differential operators and Sobolev spaces on a manifold since it is similar as in the constant order case. Details for this can be found in \cite{hintz20}.
\subsection{Microlocalization}

\begin{definition}
	Let $a\in S^m(\rr^n)$. Then a point $(x_0,\xi_0)\in \rr^n\times(\rr^n\setminus \{0\})$ does not lie in the \emph{essential support} of $a$ 
	\beq
	{ess \ supp} (a)\subset \rr_x^n\times(\rr_\xi^n\setminus \{0\})
	\eeq 
	if and only if $a$ is a symbol of order $-\infty$ near $x_0$ and in a conic neighborhood of $\xi_0$; that is, there exists $\epsilon >0$ such that for all $\alpha\in \mathbb N_0^n$, $\beta \in \mathbb N_0^n$, $k\in \rr$ we have
	\beq
	\vert \p^\alpha_x\p^\beta_\xi a(x,\xi)\vert\leq C_{\alpha\beta k}\bra\xi\ket^{-k} \ \ \forall (x,\xi), \  \vert\xi\vert\geq 1, \ \vert x-x_0\vert+\bigg\vert\frac{\xi}{\vert\xi\vert}-\frac{\xi_0}{\vert\xi_0\vert}\bigg \vert\leq \epsilon.
	\eeq
\end{definition}
\begin{remark}
By definition ess supp $a$ is closed and conic in $\xi$.
\end{remark}

\begin{definition}
Let $A={\rm Op}(a)$. Then we define the \emph{operator wave front set} of $A$ as the closed, conic set
\beq
\WF'(A):= {\rm ess \ supp}(a)\subset \rr_x^n\times(\rr_\xi^n\setminus \{0\}).
\eeq
\end{definition}

We next redefine ellipticity of operators and symbols in a microlocal manner analogous to ${\rm ess \supp}$ and $\WF'$.

\begin{definition}
	Let $A\in \Psi^m(M)$. Then the \emph{elliptic set of} $A$,
	\beq
	{\rm Ell}(A)\subset T^*M\setminus o
	\eeq
	consists of all $(x_0,\xi_0)\in \rr^n\times(\rr^n\setminus \{0\})$ in a conic neighborhood of which $\sigma_m(A)$ is elliptic; that is, in local coordinates and picking a representative of $\sigma_m(A)$, there exist $c, C>0$ and $\epsilon>0$ such that
		\beq
	\vert \sigma_m(A)(x,\xi)\vert\geq c_{\alpha\beta k}\bra\xi\ket^{-k} \ \ \forall (x,\xi), \  \vert\xi\vert\geq C, \ \vert x-x_0\vert+\bigg\vert\frac{\xi}{\vert\xi\vert}-\frac{\xi_0}{\vert\xi_0\vert}\bigg\vert\leq \epsilon.
	\eeq
	The complement of ${\rm Ell}(A)$ is the \emph{characteristic set}
	\beq
	{\rm Char}(A):=(T^*M\setminus o)\setminus {\rm Ell}(A).
	\eeq
\end{definition}

\begin{proposition}
	Let $A\in \Psi^m(M)$, and suppose $K\subset {\rm Ell}(A)$ is a closed subset. Then there exists a microlocal parametrix for $A$ on $K$, namely, an operator $B\in \Psi^{-m}(M)$ such that
	\beq
	K\cap \WF'(AB-I)=\emptyset, \ \, K\cap\WF'(BA-I)=\emptyset.
	\eeq
\end{proposition}

Now that we have defined the wave front set of an operator we can define the wave front set of a distribution.

\begin{definition}
	Let $u\in \mathcal D'(M)$ and let $S^*M$ be the unit sphere bundle in $T^*M$. Then $\alpha\in S^*M$ does not lie in the wave front set,
	\beq
	\alpha\notin \WF(u)\subset S^*M
	\eeq
	if and only if there exists a neighborhood $U\subset S^*M$ of $\alpha$ such that
	\beq
	A\in \Psi^0(M), \ \WF'(A)\subset U \implies Au\in C^\infty(M).
	\eeq
	It is closed by definition.
\end{definition}

We have an equivalent definition.
\begin{proposition}
	Let $u\in \mathcal D'(M)$. Then
	\beq
	\WF(u)=\underset{\bea A&\in\Psi^0(M), \\
		 Au &\in C^\infty(M)\eea}{\bigcap}{\rm Char}(A).
	\eeq
\end{proposition}

\subsection{Propagation of singularities and radial estimates}
Let $M$ be a differentiable manifold equipped with a volume form.

\begin{theorem}\label{propsing}
	$P,Q \in\Psi^m(M)$ for $m\in \rr$  with real homogeneous of order $m$ principal symbol $p,$ $q$ and $q=\sigma_m(Q)\geq 0$. Let $B,G,E \in \Psi^0(M)$ such that
	
	(1) $\WF(B)\subset {\rm Ell}(G)$,
	
	(2) all backward null-bicharacteristics of $P$ starting from a point in $\WF'(B)\cap {\rm Char}(P)$ enter ${\rm Ell} (E)$ in finite time while remaining in ${\rm Ell} (G)$.
	\\
	Then the following estimate holds for any $N\in \mathbb R$, $s\in S^0(M)$ non-decreasing along the Hamiltonian flow $H_p$ and a constant $C>  0$:
	
	\begin{equation}
		\Vert Bu\Vert^2_{H^s}\leq C(\Vert G(P+iQ) u\Vert^2_{H^{s-m+1}} +\Vert Eu\Vert^2_{H^s}+\Vert u\Vert^2_{H^{-N}} ).
	\end{equation}
	
	Moreover, this holds in the strong sense that if $u\in \mathcal D'(M)$ is such that the right hand side is finite, then so is the left hand side, and the estimates holds.
	
	A similar result holds replacing $non-decreasing" by "non-increasing$ and "backward" by "forward". It is obtained by replacing $P$ with $-P$.
\end{theorem}

The proof can be found in \cite{{VasBasWun}} without the complex absorption $Q$ but we will  give full a proof for the semi-classical case Theorem \ref{semipropsing}.

We recall the definitions of radial sets. Let $m\in\rr$ and $P\in \Psi^m_{\rm cl}(M)$ be a classical operator with real homogeneous principal symbol $p$. Let $\tilde p:=\rho^{m}p$ and $\tilde H_p:=\rho^{m-1}H_p$ for an elliptic symbol $0\neq \rho\in S^{-1}_{\rm cl}(T^*M)$, so that they can be extended smoothly on fiber infinity.

\begin{definition}\label{defradsets}
	Let $\cR\subset {\rm Char}(P)$ be a smooth submanifold of $\bar T^*M$ tangent to $\tilde H_p$ and such that $d\tilde p\neq 0$ in a neighborhood of $\cR$ in $S^*M$.
	
	We say that $\cR$ is a \emph{radial source} if the following properties are satisfied.
	
	(1) Suppose $\rho_{1,j}\in C^\infty(S^*M)$, $j=1,...,k$ define $\cR$ inside ${\rm Char}(P)$, in the sense that
	\beq
	\cR=\{\tilde p=0,\rho_{1,1}=0,...,\rho_{1,k}=0\}
	\eeq 
	and the $d\rho_{1,j}$ are linearly independent at $\cR$. Let
	\beq
	\rho_1=\sum_{j=1}^k\rho_{1,j}^2,
	\eeq
	which is a 'quadratic defining function' of $\cR$. Since $\tilde H_p$ is tangent to $\cR$, the derivatives $\tilde H_p\rho_{1,j}$ vanish at $\cR$, hence $\tilde H_p\rho_1$ vanishes quadratically at $\cR$. 
	
	We assume that there exists a positive function $\beta_1\in C^\infty(S^*M)$ such that
	\beq
	\tilde H_p\rho_1=\beta_1\rho_1+F_2+F_3
	\eeq
	where $F_2\geq 0$, and $F_3$ vanishes cubically at $\mathbb R$. (thus, $\mathbb R$ is a source for the $\tilde H_p$-flow within ${\rm Char}(P)\subset S^*M$ since $\vert F_3\vert\leq C\rho_1^{3/2}\leq\frac{1}{2}\beta_1\rho_1$ near $\cR$, so $\tilde H_p\rho_1$.)
	
	(2) We have
	\beq
	\tilde H_p\rho=\beta_0\rho, \ \beta_0\vert_{\cR}> 0.
	\eeq
	(The assumption implies that $\cR$ is also a source in the fiber radial direction)
	
	(3) Let $p_1:=\sigma_{m-1}(\frac{1}{2i}(P-P^*))$ and $\tilde p_1:=\rho^{m-1}p_1$. Define $\tilde \beta\in C^\infty(S^*M)$ near $\cR$ by
	\beq
	\tilde p_1=\beta_0\tilde \beta.
	\eeq
	
	We say that $\cR$ is a \emph{radial sink for $p$} if it is a radial source for $-p$.
\end{definition}

\begin{theorem}
	\label{estimeesradiales}
	Let $P, Q\in \Psi^m_\delta$ with real principal symbol and $\cR\subset {\rm Char}(P)\subset S^*M$ a radial source for $P$. Let $u\in \mathcal D'(M)$, $Pu=f$.
	
	(1)(High regularity estimate on $\cR$) Suppose $B$, $G\in \Psi^0(M)$ be such that
	
	$ \ \ \ $ (a) $\WF '(B) \subset {\rm Ell}(G)$;
	
	$ \ \ \ $ (b) ${\rm Ell}(G)$ contains a neighborhood of $\cR$;
	
	$ \ \ \ $ (c) all backward null-bicharacteristics of $P$ from $\WF'(B)\cap{\rm Char}(P)$ tend to $\cR$ while remaining in ${\rm Ell}(G)$.
	
	$ \ \ \ $ (d) $\WF'(Q)\cap {\rm Ell}(G)=\emptyset$.
	
	Then for all $s\in S^0(M)$ constant on $\WF'(G)$ and $S, \ N\in \rr$ such that $s> S> \frac{m-1}{2}+\tilde \beta$ on $\cR$, there exists $C> 0$ such that if $\WF^{S}(u)\cap\cR=\emptyset$, then
	\beq
	\norm{Bu}_{H^s}\leq C(\norm{G(P+iQ)u}_{H^{s-m+1}}+\norm{u}_{H^{-N}})
	\eeq
	
	(2) (Low regularity estimate on $\cR$)Suppose $B$, $G,$ $E\in \Psi^0(M)$ be such that 
	
	$ \ \ \ $(a) $\WF '(B) \subset {\rm{Ell}}(G)$;
	
	$ \ \ \ $(b) all forward null-bicharacteristics of $P$ from $\WF'(B)\cap{\rm Char}(P)$  are either contained in $\cR$, or enter ${\rm Ell}(E)$ in finite time, all while remaining in ${\rm Ell}(G)$.
	
	$ \ \ \ $ (c) $\WF'(Q)\cap{\rm Ell}(G)=\emptyset$.
	
	Then for all $s\in S^0(M)$ constant on $\WF'(G)$, $ N\in \rr$ such that $s< \frac{m-1}{2}+\tilde \beta$ on $\cR$, there exists $C> 0$ such that
	\beq
	\norm{Bu}_{H^s}\leq C(\norm{G(P+iQ)u}_{H^{s-m+1}}+\norm{Eu}_{H^{s}}+\norm{u}_{H^{-N}})
	\eeq
	This estimate holds in the usual strong sense.
\end{theorem}

This theorem is formulated as in \cite{Hintzmicroloc} but the complete proof can be found in \cite{VasyEst} Proposition 2.2 and 2.3. It is important to notice that while we have assumed $s$ to be constant on the subset we are interested in we are still working with $\Psi_\delta^\infty$ and $\delta\neq 0$. The proof of \cite{VasyEst} is made to avoid the sharp Garding inequality which is where the proof would break down but it is not explicitly formulated in a way that includes $\delta\neq 0$.
This is however not important we could choose to use $B$, $G$ and $E$ in $\Psi_0^\infty$, the $\langle D\rangle^s$ composed with any of those operators is in $\Psi_0^\infty$ and the proof directly applies. Even if it used the sharp Garding inequality.

\section{Semiclassical analysis for variable order}\label{Sem}

In this section we follow the introduction of semi-classical pseudo differential operators by \cite{DyaZwo} that we adapted to the case of variable order.

\subsection{Semi-classical pseudo-differential operators}
We will first introduce semiclassical differential operators. The algebra of semiclassical operators of order $k$ contains elements of the form
\beq
A=\sum_{\vert\alpha\vert\leq k}\sum_{j=0}^{k-\vert\alpha\vert}h^ja_{\alpha j}(x)(hD_x)^\alpha:C^\infty(M)\to C^\infty(M)
\eeq
where $\alpha$ is a multiindex, the $a_{\alpha j}$ are smooth coefficients on $M$ and $D_x=\frac{1}{i}\p_x$. $h$ is the semiclassical parameter that will be taken smaller as needed or even tend to $0$ to get high frequency phenomena.

\begin{definition}
	Let $\rm m\in S^0(\rr^n;\rr^N)$, $n,N\in \mathbb N$ and $\delta>0$. Then the space of symbols of order $\rm m$ 
	\beq
	S^{\rm m}_{h,\delta}(\rr^n;\rr^N)\subset C^\infty(\rr^n\times \rr^N)
	\eeq
	consists of all functions $a=a(x,\xi)$ which for all $\alpha\in \mathbb R^n_0, \beta\in \mathbb R^N_0$ satisfy the estimate
	\beq
	\vert \p^\alpha_x\p^\beta_\xi a(x,\xi)\vert\leq C_{\alpha,\beta}\langle \xi\rangle^{\rm m-\vert \beta\vert+ \delta(\vert\alpha\vert+\vert\beta\vert)}
	\eeq
	for some constants $C_{\alpha,\beta}$. We equip this space with the family of semi-norms
	\beq
	\norm{a}_{{\rm m},k}:= \sup_{(x,\xi)\in \rr^n\times\rr^N}\max_{\vert \alpha\vert+\vert\beta\vert\leq k}\langle\xi\rangle^{-{\rm m}+\vert \beta\vert-\delta(\vert\alpha\vert+\vert\beta\vert)}\vert \p^\alpha_x\p^\beta_\xi a(x,\xi)\vert
	\eeq
	for $k\in\rr$.
\end{definition}

We also have, analogously to the constant order case, that the multiplication satisfies
\beq
S^{\rm m}_{h,\delta}(\rr^n;\rr^N)\times S_{h,\delta}^{\rm m'}(\rr^n;\rr^N)\to S_{h,\delta}^{\rm m+\rm m'}(\rr^n;\rr^N)
\eeq
and the derivation satisfies
\beq
\p^\alpha_x\p^\beta_\xi: S^{\rm m}_{h,\delta}(\rr^n;\rr^N)\to S^{\rm m-\vert\beta\vert}_{h,\delta}(\rr^n;\rr^N).
\eeq
\begin{definition}
	Let ${\rm m}\in S^{0}(\rr^n;\rr^N)$. A symbol $a\in S^{\rm m}_{h,\delta}(\rr^n;\rr^N)$ is said to be elliptic if there exists a symbol $b\in S^{-\rm m}_{h,\delta}(\rr^n;\rr^N)$ such that $ab-1\in S^{-1}_{h,\delta}(\rr^n;\rr^N)$.
\end{definition}

We can now define the quantization of symbols of variable order by
\beq
(\Op_h(a)u)(x)=(2\pi)^{-n}\int_{\rr^n}\int_{\rr^N}e^{\frac{i}{h}(x-y).\xi}a(x,\xi)u(y)dyd\xi, \ \ u\in\mathcal S(\rr^n)
\eeq
where $\mathcal S(\rr^n)$ is the space of Schwartz functions.

The space of semi-classical pseudo-differential operator of variable order $\rm m$ is $\Psi_{h,\delta}^{\rm m}(\rr^n):=\Op(S_\delta^{\rm m}(\rr^n;\rr^n))$ and it can be shown that the two spaces are isomorphic as Fréchet spaces with the induced topology. 

In the following we will omit the subscript $\delta$ for simplicity but not the $h$ to differentiate semi-classical calculus.

\begin{definition}
	We define the symbol $\sigma(A)$ of a pseudo-differential operator $A\in \Psi_h^{\rm m}(\rr^n)$ as the unique symbol such that $\Op(\sigma(A))=A$.
	
	The semi-classical principal symbol $\sigma_{h}(A)$ of $A={\rm Op}(a)\in \Psi_h^{\rm m}(\rr^n)$ is defined somewhat differently. We write the expansion in $h\to 0$
	\beq
	a(x,\xi,h)\sim \sum_{j=0}^\infty h^j aj(x,\xi,h).
	\eeq
	Then 
	\beq
	\sigma_h(A):=a_0.
	\eeq
\end{definition}

\begin{remark}
	Roughly speaking taking the semi-classical symbol is taking $h=0$ except if it multiplies $D_x$, then $hD_x\to \xi$. The usual example is the Klein-Gordon operator on Minkowski $\square=D_t^2-D_x^2-\lambda^2$. Its principal symbol in term of microlocal analysis is $\sigma_2(\square)=\tau^2-\xi^2$ and its characteristic set are the light cones $\tau^2-\xi^2=0$. However we can make this operator into a semi-classical one by taking $P=h^2\square$ and $\omega=h \lambda$. We get $\sigma_h(P)=\tau^2-\xi^2-\omega^2$, its characteristic set, that we will define analogously later on, are then the hyperboles of finite energy $\tau^2-\xi^2-\omega^2=0$. This is why we say semi-classical analysis captures the behavior of operator at high energies, here $\lambda =\omega/h$ with $h\to 0$. If we had taken $\lambda$ to be independent of $h$ we would not have seen this high energy behavior and the semi-classical limit (the symbol) would have just become the microlocal limit.
	\end{remark}

\begin{remark}
	The main difference between semi-classical calculus and pseudo-differential calculus is that the asymptotic sum are take for $h\to 0$ instead of $\xi\to \infty$. However, since in the quantization procedure we have $\frac{\xi}{h}$, play a similar role but it does not reveal the same information about the symbol.
	\end{remark}
	
	\subsection{Sobolev spaces}
	
	Similarly to the microlocal case, the operators act on Sobolev spaces.
	\begin{lemma}
		Let $A\in \Psi_{h,\delta}^{0}(\rr^n)$, $A$ is a bounded operator on $L^2(\rr)$.
	\end{lemma}
	
	We can now define Sobolev spaces of variable order.
	
	\begin{definition}\label{semsob}
		Let $A\in \Psi_{h,\delta}^{\rm m}(\rr^n)$ be an elliptic operator. We define the Sobolev spaces of variable order $\rm m$ as the subset of $L^2(\rr^n)$ such that 
		\beq
		\norm{u}_{H_h^{\rm m}(\rr^n)}:=\norm{Au}_{L^2(\rr^n)}<+\infty
		\eeq
			To fix a definition we take $A:=\langle hD\rangle^m$. Where we define $\langle hD\rangle^{\rm s}$ as ${\rm Op}_h(\langle \xi\rangle^{-\rm s})$.
	\end{definition}
	
	For simplicity we will omit the $\delta$ in the rest of this section.
	
	\begin{proposition}\label{semsob}
		Let ${\rm s,m}\in S^0(\rr^n,\rr^n)$ and $A\in \Psi^{m}(\rr^n)$, then $A:H_h^{\rm s}(\rr^n)\to H_h^{\rm s-\rm m}(\rr^n)$ is bounded.
		
			The definition of the Sobolev spaces does not depend on $A$ and the norms are equivalent for any $A$.
	\end{proposition}

	We also define $H_{h,\rm loc}^{\rm s}(\rr^n)$ as distributions that are locally in $H_h^{\rm s}(\rr^n)$ and $H_{h,\rm comp}^{\rm s}(\rr^n)$ as elements of $H_{h,\rm loc}^{\rm s}(\rr^n)$ with compact support.
	
	We will omit the construction of pseudo-differential operators and Sobolev spaces on a manifold since it is similar as in the constant order case. Details can be found in \cite{DyaZwo}

	\subsection{Microlocalization}
	
	\begin{definition}
		Let $a\in S^m(\rr^n)$. Then a point $(x_0,\xi_0)\in \rr^n\times(\rr^n\setminus \{0\})$ does not lie in the \emph{essential support} of semi-classical symbol  $a$ 
		\beq
		{ess \ supp} (a)\subset \rr_x^n\times(\rr_\xi^n\setminus \{0\})
		\eeq 
		if and only if $a$ is a symbol of order $-\infty$ near $x_0$ and in a  neighborhood of $\xi_0$; that is, there exists $\epsilon >0$ such that for all $\alpha\in \mathbb N_0^n$, $\beta \in \mathbb N_0^n$, $k\in \rr$, $N\in \mathbb N$ we have
		\beq
		\vert \p^\alpha_x\p^\beta_\xi a(x,\xi)\vert\leq C_{\alpha\beta kN}h^N\bra\xi\ket^{-k} \ \ \forall (x,\xi), \  \vert\xi\vert\geq 1, \ \vert x-x_0\vert+\vert{\xi}-{\xi_0}\vert\leq \epsilon.
		\eeq
	We say that if ${\rm ess \ sup} (a)=\emptyset$ then $a$ is in the \emph{residual class} $h^\infty S^{-\infty}_h(\rr^n)$.
	\end{definition}
	\begin{remark}
		In this case the essential support is not a conic set but it is still closed. The reason for this is that we may still have decay in $h$ even at finite $\xi$.
	\end{remark}
	
	\begin{definition}
		Let $A={\rm Op}_h(a)$. Then we define the \emph{semi-classical operator wave front set} of $A$ as the closed set
		\beq
		\WF_h'(A):= {\rm ess \ supp}(a)\subset \rr_x^n\times(\rr_\xi^n\setminus \{0\}).
		\eeq
	\end{definition}
	Then the following properties hold true.

\begin{definition}
	Let $A\in \Psi^m_h(M)$. Define the set
	\beq
	{\rm ell}_h(A)=\{(x_0,\xi_0)\in \bar T^*M\st (\bra\xi\ket^{-m}\sigma_h(A))(x_0,\xi_0)\neq 0\}.
	\eeq
	If $(x_0,\xi_0)\in {\rm ell}_h(A)$ we say $A$ is elliptic at $(x_0,\xi_0)$.
\end{definition}

\begin{definition}
We say that $A\in \Psi^m_h(M)$ (or $A\in \Psi^m(M)$), with Schwartz kernel $K$, is \emph{properly supported} if the projection of the maps $\pi_L: \supp K\to M$ and $\pi_R: \supp K\to M$ proper, i.e. preimages of compact sets are compact.
\end{definition}

\subsection{Semi-classical estimates}

\begin{theorem}\label{semipropsing}
	Let $P\in \Psi_{h,\delta}^m(M)$ with $m\in \rr$ with real semi-classical symbol $p:=\sigma_h(P)$. Let $B,G,E \in \Psi_h^0(M)$ such that
	
	(1) $\WF_h(B)\subset {\rm Ell}_h(G)$,
	
	(2) all backward null-bicharacteristics of $P$ starting from a point in $\WF_h'(B)\cap {\rm Char}_h(P)$ enter ${\rm Ell}_h (E)$ in finite time while remaining in ${\rm Ell}_h (G)$.
	\\
	Then the following estimate holds for any $N\in \mathbb R$, $s\in S^0(M)$ non decreasing along the Hamiltonian flow $H_p$ and a constant $C>  0$:
	\begin{equation}
		\Vert Bu\Vert^2_{H_h^s}\leq C(h^{-1}\Vert GPu\Vert^2_{H_h^{s-m+1}} +\Vert Eu\Vert^2_{H_h^s}+h^N\Vert u\Vert^2_{H_h^{-N}} ).
	\end{equation}
	
	Moreover, this holds in the strong sense that if $u\in \mathcal D'(M)$ is such that the right hand side is finite, then so is the left hand side, and the estimates holds.
	
\end{theorem}

\begin{remark}
	We can localize the problem in the following way: take $\chi, \tilde \chi\in \cinf_\c(M)$ with $\tilde \chi=1$ on $supp \chi$, we have that 
	\[\chi P\tilde\chi=\chi P+\chi[P,\tilde \chi]\]
	
	So $\sigma_h(\chi P\tilde\chi)=\chi \sigma_h(P)$. ${\rm Ell}_h$, ${\rm Char}_h$ and null-bicharacteristics are all defined from the semi-classical principal symbol so they are identical for the localized and non-localized operator.
	
	Moreover, for $u\in \mathcal D'(M)$,
	\[\chi P\tilde\chi u-\chi Pu=\chi[P,\tilde \chi]u \in \cinf_\c(M)\]
	So $\WF_h'(\chi P\tilde\chi)\cap \chi^{-1}(\{1\})=\WF_h'(P)\cap\chi^{-1}(\{1\})$. Therefore if it is equivalent for the condition of the theorem to hold for the globally or locally everywhere. Using a partition of unity, we can recover the full estimates if we only have localized estimates everywhere so we only have to prove the theorem for operators with compact Schwartz kernel. Similarly if its true for all $u$ with compact support it holds for all $u$. 
\end{remark}


\begin{definition}\label{defradsets}
	Let $\cR\subset {\rm Char}(P)$ be a smooth submanifold of $\bar T^*M$ tangent to $\tilde H_p$ and such that $d\tilde p\neq 0$ in a neighborhood of $\cR$ in $S^*M$.
	
	We say that $\cR$ is a \emph{radial source} if the following properties are satisfied.
	
	(1) Suppose $\rho_{1,j}\in C^\infty(S^*M)$, $j=1,...,k$ define $\cR$ inside ${\rm Char}(P)$, in the sense that
	\beq
	\cR=\{\tilde p=0,\rho_{1,1}=0,...,\rho_{1,k}=0\}
	\eeq 
	and the $d\rho_{1,j}$ are linearly independent at $\cR$. Let
	\beq
	\rho_1=\sum_{j=1}^k\rho_{1,j}^2,
	\eeq
	which is a 'quadratic defining function' of $\cR$. Since $\tilde H_p$ is tangent to $\cR$, the derivatives $\tilde H_p\rho_{1,j}$ vanish at $\cR$, hence $\tilde H_p\rho_1$ vanishes quadratically at $\cR$. 
	
	We assume that there exists a positive function $\beta_1\in C^\infty(S^*M)$ such that
	\beq
	\tilde H_p\rho_1=\beta_1\rho_1+F_2+F_3
	\eeq
	where $F_2\geq 0$, and $F_3$ vanishes cubically at $\mathbb R$. (thus, $\mathbb R$ is a source for the $\tilde H_p$-flow within ${\rm Char}(P)\subset S^*M$ since $\vert F_3\vert\leq C\rho_1^{3/2}\leq\frac{1}{2}\beta_1\rho_1$ near $\cR$, so $\tilde H_p\rho_1$.)
	
	(2) We have
	\beq
	\tilde H_p\rho=\beta_0\rho, \ \beta_0\vert_{\cR}> 0.
	\eeq
	(The assumption implies that $\cR$ is also a source in the fiber radial direction)
	
	(3) Let $p_1:=\sigma_h(\frac{1}{2ih}(P-P^*))$ and $\tilde p_1:=\rho^{m-1}p_1$. Define $\tilde \beta\in C^\infty(S^*M)$ near $\cR$ by
	\beq
	\tilde p_1=\beta_0\tilde \beta.
	\eeq
	
	We say that $\cR$ is a \emph{radial sink for $p$} if it is a radial source for $-p$.
\end{definition}

\begin{theorem}\label{semradest}
	Let $P, Q\in \Psi^m_{h,\delta}$ with real principal symbol and $\cR\subset {\rm Char}_h(P)\subset S^*M$ a radial source for $P$. Let $u\in \mathcal D'(M)$, $Pu=f$.
	
	(1)(High regularity estimate on $\cR$) Suppose $B$, $G\in \Psi_{h,\delta}^0(M)$ be such that
	
	$ \ \ \ $ (a) $\WF_h '(B) \subset {\rm Ell}_h(G)$;
	
	$ \ \ \ $ (b) ${\rm Ell}_h(G)$ contains a neighborhood of $\cR$;
	
	$ \ \ \ $ (c) all backward null-bicharacteristics of $P$ from $\WF_h'(B)\cap{\rm Char}_h(P)$ tend to $\cR$ while remaining in ${\rm Ell}_h(G)$;
	
	$ \ \ \ $ (d) $\WF_h'(Q)\cap {\rm Ell}_h(G)=\emptyset$.
	
	Then for all $s\in S^0(M)$ constant on $\WF_h'(G)$ and $S, \ N\in \rr$ such that $s> S> \frac{m-1}{2}+\tilde \beta$ on $\cR$, there exists $C> 0$ such that if $\WF_h^{S}(u)\cap\cR=\emptyset$, then
	\beq
	\norm{Bu}_{H_h^s}\leq C(h^{-1}\norm{G(P+iQ)u}_{H_h^{s-m+1}}+h^N\norm{u}_{H_h^{-N}}).
	\eeq
	
	(2) (Low regularity estimate on $\cR$)Suppose $B$, $G,$ $E\in \Psi_{h,\delta}^0(M)$ be such that 
	
	$ \ \ \ $(a) $\WF_h '(B) \subset {\rm{Ell}}_h(G)$;
	
	$ \ \ \ $(b) all forward null-bicharacteristics of $P$ from $\WF_h'(B)\cap{\rm Char}_h(P)$  are either contained in $\cR$, or enter ${\rm Ell}_h(E)$ in finite time, all while remaining in ${\rm Ell}_h(G)$.
	
	$ \ \ \ $ (c) $\WF_h'(Q)\cap{\rm Ell}_h(G)=\emptyset$.
	
	Then for all $s\in S^0(M)$ constant on $\WF_h'(G)$, $ N\in \rr$ such that $s< \frac{m-1}{2}+\tilde \beta$ on $\cR$, there exists $C> 0$ such that
	\beq
	\norm{Bu}_{H_h^s}\leq C(h^{-1}\norm{G(P+iQ)u}_{H_h^{s-m+1}}+\norm{Eu}_{H_h^{s}}+h^N\norm{u}_{H_h^{-N}}).
	\eeq
	This estimate holds in the usual strong sense.
\end{theorem}

A proof of the constant order case, which applies verbatim in the variable order case, can be found in \cite{VasyEst}.

\section{De Sitter space}\label{deS}

\subsection{De Sitter geometry}\label{deSitter}
 We will first present the de Sitter space of which we will later on take perturbations of. In this case it is easy to see that the wave operator is essentially selfadjoint.
 
 The de Sitter space is a solution or, a family of solutions when rescaled properly, of the Einstein equations with a positive cosmological constant $\Lambda$. For the following normalization $\Lambda$ will be taken to be $\frac{(n-1)(n-2)}{2}$ for simplicity. It is the only one that is symmetric under the Lorentz group. It is considered to have particular relevance for cosmological models.
 
  Oftentimes only a region of the de Sitter space called the static patch is studied. Here we are interested in the whole space as we will see in the following. 
 
 We first give the definition of de Sitter space as an immersed hypersurface in a higher dimensional Minkowski space. 
 
\begin{definition}
In Minkowski space $\rr ^{1+d}$, $g_M=dz_0^2-(dz_1^2+...+dz_d^2)$, we define the de Sitter space $\rm dS$ as
\beq
{\rm dS}:=\{z_0^2-(z_1^2+...+z_d^2)=-1\}.
\eeq
We also define the hyperbolic $d$ dimensional spaces $\mathbb H^d_\pm$ as
\beq
\mathbb H^d_\pm=\{z_0^2-(z_1^2+...+z_d^2)=1, \ \pm z_0> 0\}.
\eeq
Both are equipped with the metric induced by $g_M$.
\end{definition}

In the de Sitter case, the induced metric is Lorentzian, whereas in the hyperbolic space case it is Riemannian.

We want to compactify $\rm dS$  into a compact manifold with boundary $\overline{\rm dS}$. To that end we take $x_j=\frac{z_j}{\rho(z)}\in \mathbb S^d$ where $\rho(z)=\sqrt{\sum_{0\leq j\leq d+1} z_j^2}$. The mapping $(z_j)_j\mapsto(x_j)_j$ defines a diffeomorphism between $\rm dS$ and $\{x_0^2> x_0^2+...+x_n^2\}\cap \mathbb S^d$ and between $\mathbb H^d_\pm$ and $\{\pm x_0< 0\}\cap\{x_0^2> x_0^2+...+x_n^2\}\cap \mathbb S^d$.

We now define $\mathcal I_\pm:=\{\pm x_0>0\}\cap\{x_0^2=x_0^2+...+x_n^2\}\cap \mathbb S^d$ which are both topologically $\mathbb S^{d-1}$ and  $\mathcal I:=\mathcal I_+\cup\mathcal I_-$. This gives us an embedding of $\rm dS$ and $\mathbb H^d_\pm$ into compact manifolds with boundary $\p\overline{{\rm dS}}=\mathcal I$ and $\p\overline{{\mathbb H_\pm^d}}=\mathcal I_\pm$ respectively. 

\begin{proposition}
In the neighborhood of $\partial \overline{\rm dS}\cap \{\pm x_0> 0\}$ we have a coordinate system $(y_0, \Omega)\in [0,\infty)\times {\mathbb S^{d-1}}$ for which  the metric is of the form
\beq
g=\frac{dy_0^2-g_0(y_0,\Omega_j,d\Omega_j)}{y_0^2}, \ \ g_0=\frac{(1+y_0)^2}{4}d\Omega_{\mathbb S^{d-1}};
\eeq
here $y_0=0$ corresponds to one of the two connected components of $\partial\overline{\rm dS}$.
\end{proposition}

\begin{proof}
Take $\sinh h =z_0$ and $z_j=\cosh h \,\Omega_j$ where $(\Omega_j)_j$ is the coordinate system of the (d-1)-dimensional sphere.

We have $\cosh(h)dh=dz_0$ and $dz_j=\sinh(h)\Omega_jdh +\cosh(h)d\Omega^j$, hence
\beq
g=(\cosh^2 h -\sinh^2 h)dh^2-\sinh h \cosh h \, dh\sum_j\Omega_jd\Omega_j+\cosh^2 h d\Omega_{\mathbb S^{d-1}}^2,
\eeq
 but $\sum_j\Omega_jd\Omega_j=d\sum_j \Omega_j^2=0$ so 
\beq
g=dh^2-\cosh^2 h d\Omega_{\mathbb S^{d-1}}^2.
\eeq
For $\pm h> 0$ we take $y_0$ such that $y_0=e^{\mp h}$ so $dy_0=\mp y_0 dh$ and
\beq
g=\frac{dy_0^2-g_0(y_0,\Omega_j,d\Omega_j)}{y_0^2}, \ \ g_0=\frac{(1+y_0^2)^2}{4}d\Omega_{\mathbb S^3}.
\eeq
\end{proof}
\begin{remark}
The metric $y_0^2g$ can be extended to a smooth metric on $\overline{\rm dS}$ and we have $\vert dy_0\vert_{y_0^2g}=1$ on $\partial \overline{\rm dS}$.
\end{remark}

\subsection{Conformal extension}

In this section we follow closely the construction of \cite{WroExa}, following itself the work of \cite{VasyEst}.

We start by considering $1+d$-dimensional Minkowski space $\mathbb M^{1,d}=\rr^{1+d}$ equipped with its canonical flat metric
\[
g_\mathbb M=dx_0^2 - (d x_1^2 + \cdots +d x_d^2).
\]
We use the convention 
\[
x\cdot y = x_0 y_0 - x_1 y_1 - \cdots - x_d y_d
\]
for the Minkowskian product, and extend its definition to $\cc^{d+1}$ (as a bi-linear form), the complexification of $\mm^{1,d}$.

We denote by $\rho$ the Euclidean distance from the origin, i.e.
\[
\rho = (x_0^2+\cdots + x_d^2)^{\12}.
\]
We also consider the Minkowski distance function
\[
r= \left| x_0^2 -(x_1^2+\cdots +x_d^2)\right|^{\12}.
\]
For our purposes, $d$-dimensional de Sitter space ${\rm dS}^d$ (of radius $r>0$) is best defined as the hyperboloid 
\beq\label{hyp1}
{\rm dS}^d = \{ x\in \rr^{1+d} \ |  \ r(x)={\rm const}, \ x_{0}^2<x_1^2+\cdots + x_d^2 \},  
\eeq
We equip ${\rm dS}$ with the metric $g_{\rm dS}$ induced from $g_\mathbb M$. 

We equip ${\mathbb H}_\pm^d$ with the Riemannian metric defined as minus the metric induced from $g_\mathbb M$.  When the embedding of hyperbolic space as a hypersurface in $d+1$ dimensional Minkowski space is not be required or the difference between ${\mathbb H}_+^d$ and ${\mathbb H}_-^d$ is not relevant, we write ${\mathbb H}^d$ instead of ${\mathbb H}_\pm^d$.

We write $x=(x_0,\ldots,x_{d})$ for points in $\mathbb M^{1,d}$, and we use alternatively the notation $x_{\rm dS}$, resp.~$x_{\mathbb H}$ when $x\in {\rm dS}^d$, resp.~$x\in{\mathbb H}^d$. For $x\in \mathbb M^{1,d}\setminus\{0\}$ we use the notation $x_\sss$ to denote the point on the unit sphere
\beq\label{eq:unitsphere}
\sss^d = \{ x\in \rr^{1+d} \ |  \ \rho(x)=1\}
\eeq
colinear to (with the same angular coordinates) $x$ (i.e., $x_\sss$ is the point where the half-line from the origin passing through $x$ intersects the sphere, see Fig.~1). 
\def\svgwidth{7cm}

\begin{figure}[h]\label{fig1}
	\begingroup%
	\makeatletter%
	\providecommand\color[2][]{%
		\errmessage{(Inkscape) Color is used for the text in Inkscape, but the package 'color.sty' is not loaded}%
		\renewcommand\color[2][]{}%
	}%
	\providecommand\transparent[1]{%
		\errmessage{(Inkscape) Transparency is used (non-zero) for the text in Inkscape, but the package 'transparent.sty' is not loaded}%
		\renewcommand\transparent[1]{}%
	}%
	\providecommand\rotatebox[2]{#2}%
	\ifx\svgwidth\undefined%
	\setlength{\unitlength}{320.85bp}%
	\ifx\svgscale\undefined%
	\relax%
	\else%
	\setlength{\unitlength}{\unitlength * \real{\svgscale}}%
	\fi%
	\else%
	\setlength{\unitlength}{\svgwidth}%
	\fi%
	\global\let\svgwidth\undefined%
	\global\let\svgscale\undefined%
	\makeatother%
	\begin{center}
	\begin{picture}(1,1.00599969)%
		\put(0,0){\includegraphics[width=\unitlength]{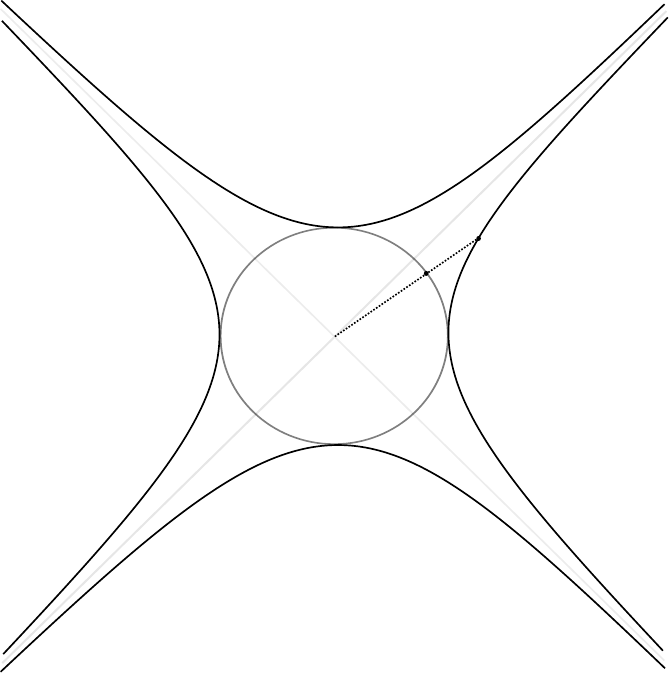}}%
		\put(0.30998022,0.75006664){\color[rgb]{0,0,0}\makebox(0,0)[lb]{\smash{${\mathbb H}_+^d$}}}%
		\put(0.31354215,0.23111462){\color[rgb]{0,0,0}\makebox(0,0)[lb]{\smash{${\mathbb H}_-^d$}}}%
		\put(0.21558804,0.3682502){\color[rgb]{0,0,0}\makebox(0,0)[lb]{\smash{${\rm dS}^d$}}}%
		\put(0.73029229,0.35578347){\color[rgb]{0,0,0}\makebox(0,0)[lb]{\smash{${\rm dS}^d$}}}%
		\put(0.71960641,0.6353978){\color[rgb]{0,0,0}\makebox(0,0)[lb]{\smash{$x$}}}%
		\put(0.6145284,0.61224503){\color[rgb]{0,0,0}\makebox(0,0)[lb]{\smash{$x_\sss$}}}%
	\end{picture}%
	\end{center}
	\endgroup%
	\caption{The map $x\mapsto x_\sss\in\sss^d$, here for $x\in{\rm dS}^d$. In reality ${\rm dS}$ is connected for $d\geq 2$.}
\end{figure}
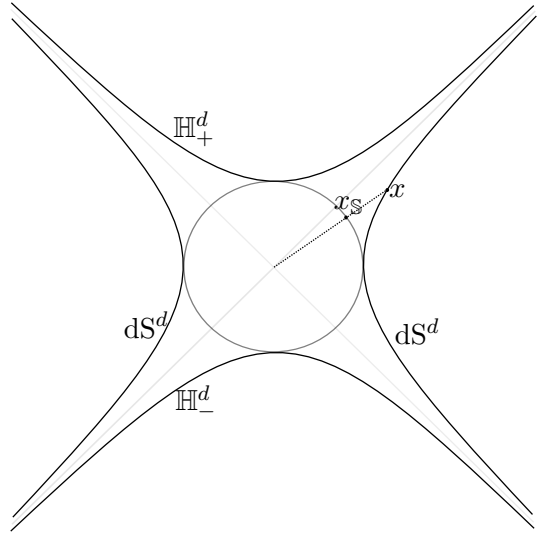

With the map $x\mapsto x_\sss\in\sss^d$ we can distinguish three regions of the sphere $\sss^d$ that we identify with the three hyperboloids ${\mathbb H}_+^d$, ${\rm dS}^d$ and ${\mathbb H}_-^d$ , namely,   resp.~$\{x_0^2>x_1^2+\cdots+x_d^2\}\cap\{ x_0 > 0\}$, $\{x_0^2<x_1^2+\cdots+x_d^2\}$ and $\{x_0^2>x_1^2+\cdots+x_d^2\}\cap\{ x_0 < 0\}$ all intersected with \eqref{eq:unitsphere}. Using these identifications we have that, ${\mathbb H}^d_+$, ${\rm dS}^d$, and ${\mathbb H}^d_-$ cover the whole $\sss^d$ at the exception the two regions that  are both  topologically $\sss^{d-1}$
\[
{Y}_\pm\defeq \sss^d \cap \mathcal{C}_\pm^d, 
\]
 where $\mathcal{C}_\pm^d$ denotes the future and past lightcone in $\mathbb M^{1,d}$, i.e.
\[
\mathcal{C}_\pm^d=\mathcal{C}^d\cap\{ \pm x_0>0 \} , \  \ \mathcal{C}^d=\{ x\in \mathbb M^{1,d} \ |  \ r(x)=0\}.
\]
We can then define a manifold with boundary $\bar{\rm dS}^d$, by endowing ${\rm dS}^d$ (as a part of $\sss^d$) with the boundary ${Y}={Y}_+\cup {Y}_-$. ${Y}_+$, resp.~${Y}_-$, is called the \emph{future}, resp.~\emph{past conformal boundary} of $\bar{\rm dS}^d$ and are connected.  ${Y}$ is  called the \emph{conformal boundary} (or \emph{conformal infinity}).  In an similar way we also obtain a manifold with boundary $\bar{\mathbb H}^d_\pm$ by taking $\mathbb H^d_\pm$ to be its interior, and ${\mathcal I}_\pm$ to be its boundary. 

\def\svgwidth{6cm}
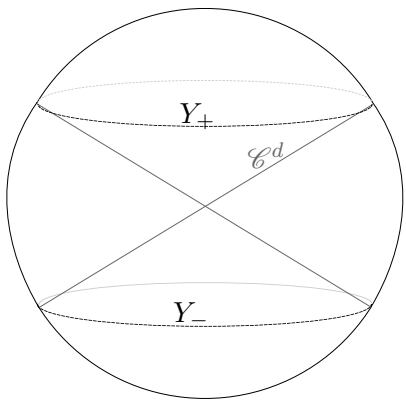
\begin{figure}[h]\label{fig2}
	\begingroup%
	\makeatletter%
	\providecommand\color[2][]{%
		\errmessage{(Inkscape) Color is used for the text in Inkscape, but the package 'color.sty' is not loaded}%
		\renewcommand\color[2][]{}%
	}%
	\providecommand\transparent[1]{%
		\errmessage{(Inkscape) Transparency is used (non-zero) for the text in Inkscape, but the package 'transparent.sty' is not loaded}%
		\renewcommand\transparent[1]{}%
	}%
	\providecommand\rotatebox[2]{#2}%
	\ifx\svgwidth\undefined%
	\setlength{\unitlength}{456.00001221bp}%
	\ifx\svgscale\undefined%
	\relax%
	\else%
	\setlength{\unitlength}{\unitlength * \real{\svgscale}}%
	\fi%
	\else%
	\setlength{\unitlength}{\svgwidth}%
	\fi%
	\global\let\svgwidth\undefined%
	\global\let\svgscale\undefined%
	\makeatother%
	\begin{center}
	\begin{picture}(1,0.86661182)%
		\put(0,0){\includegraphics[width=\unitlength]{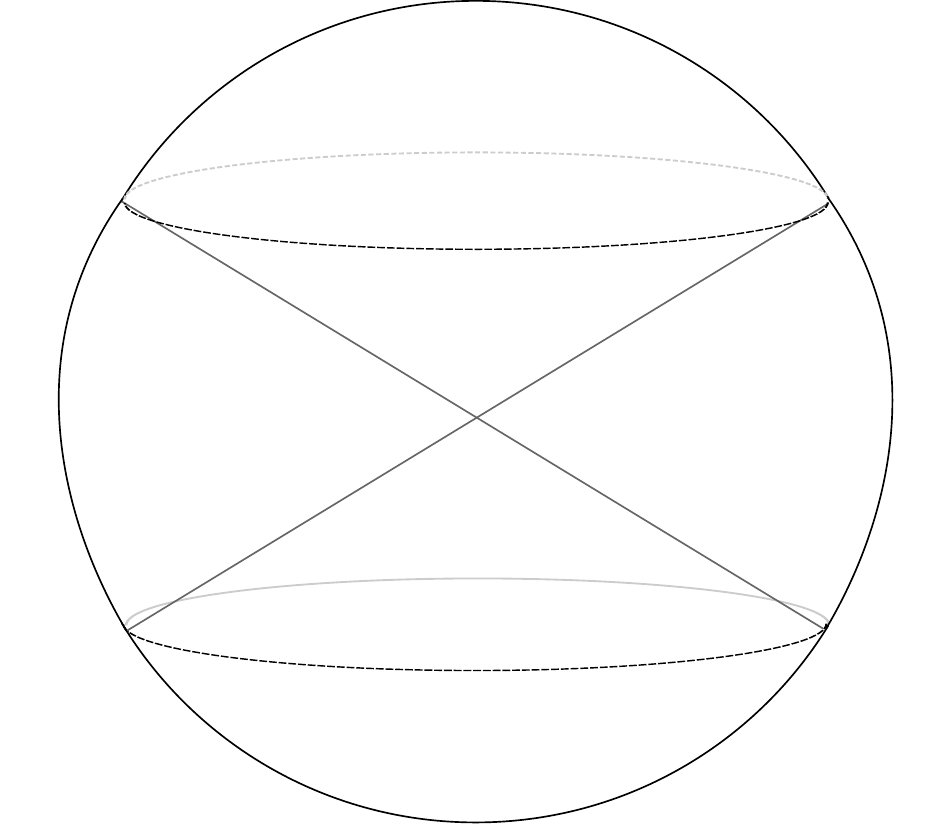}}%
		\put(0.44477831,0.61152685){\color[rgb]{0,0,0}\makebox(0,0)[lb]{\smash{${Y}_+$}}}%
		\put(0.43077883,0.17016066){\color[rgb]{0,0,0}\makebox(0,0)[lb]{\smash{${Y}_-$}}}%
		\put(0.59523805,0.51233455){\color[rgb]{0.4,0.4,0.4}\makebox(0,0)[lb]{\smash{$\mathcal{C}^d$}}}%
	\end{picture}%
	\end{center}
	\endgroup%
	\caption{Extended de Sitter space. The `equatorial belt' part of the sphere $\sss^d$ (i.e., the region between ${Y}_-$ and ${Y}_+$) is identified with ${\rm dS}^d$, and the two `caps' resp.~below ${Y}_-$ and above ${Y}_+$ are identified with resp.~${\mathbb H}^d_-$ and ${\mathbb H}^d_+$.}
\end{figure}
This way, 
\beq\label{eq:geom}
\sss^d = \bar{{\mathbb H}}^d_+ \cup \bar{\rm dS}^d \cup \bar{\mathbb H}^d_-, 
\eeq
where $\p\bar{{\mathbb H}}^d_\pm$ becomes  ${Y}_\pm$, the component  of the boundary of $\bar{\rm dS}^d$ (see Fig.~2).

\section{Asymptotically de Sitter spaces}\label{Asy}

\subsection{Introduction} 
In this section we will introduce the notion of asymptotically de Sitter spaces that are spaces that not only look like the de Sitter space at infinity but can be extended past infinity in a similar way. This construction will enable us to modify conformally our Laplace--Beltrami operator to get more well behaved one that can be studied through microlocal methods.

In this section we adapt the work of Dyatlov and Zworski, \S7 of \cite{DyaZwo}, to the Lorentzian case instead of the Riemannian one, namely, we replace de Sitter space with  hyperbolic space. This section should be read in parallel of it for proofs and missing details.

This construction originally comes from Vasy in \cite{VasyEst} where he noticed the de Sitter space can be seen a the `equatorial belt' of a sphere from which we have removed two hyperbolic spaces.

The idea is to consider perturbations of de Sitter that have an asymptotic behavior similar enough that the picture of de Sitter being part of the boundary of the ball in Minkowski with complement being its boundary and two hyperbolic spaces still holds for those perturbations.

\subsection{Geometric structure} Let us first introduce a few preliminaries on manifolds with boundary.
\begin{definition}
Let $\bar M$ be a compact manifold of dimension $d$ with boundary $\partial \bar{M}$ and interior $M$. A \emph{boundary defining function} on $\bar M$ is a $C^\infty$ function
\beq
y_0: \bar M\to [0,\infty)
\eeq
such that $y_0=0$ on $\partial \bar{M}$, $dy_0\neq 0$ on $\partial \bar{M}$, and $y_0>0$ on $M$.
\end{definition}
\bep\label{bdfeq}
Such a function always exists and two boundary defining function $y_0$, $\tilde y_0$ are multiple of each other:
\beq
\tilde y_0=e^fy_0, \ \ f\in C^\infty(\bar M,\mathbb R)
\eeq
\eep

This gives rise to a \emph{product structure}: namely, for any boundary defining function $y_0$ and $\epsilon_0> 0$ small enough, there exists a diffeomorphism
\beq\label{prodec}
y_0^{-1}([0,\epsilon_0))\to [0,\epsilon_0)\times \partial \bar{M}, \ \ x\mapsto(y_0(x),\tilde y), \ \ \tilde y\vert_{\partial \bar{M}}=I.
\eeq

\begin{definition}\label{defman}
An \emph{asymptotically de Sitter manifold} is a complete Lorentzian manifold $(M,g)$ such that:

\ben
\item  $M$ is the interior of a compact manifold with boundary, denoted $\bar M$.
\item For a boundary defining function $y_0$, the metric $y_0^2g$ extends to a smooth Lorentzian metric on $\bar M$.
\item We have $\vert dy_0\vert_{y_0^2g}=1$ on ${\partial \bar{M}}$.
\item The boundary $\p \bar{M}$ splits into two submanifolds $Y_+$ and $Y_-$ which are unions of connected components of $\p \bar{M}$.
\item  All null (causal) geodesics of $y_0^2g$, which are reparametrizations of the geodesics of $g$, tend  to $Y_+$ if they are they are future oriented and to $Y_-$ if they are past oriented. 
\een
We call  (4) and (5)  the \emph{non-trapping} (\emph{at all energies}) condition.
\end{definition}

Note that since $f\in C^\infty(\bar M,\mathbb R)$, the condition (2) does not depend on the choice of boundary defining function. Similarly (3) does not depend on this choice since
\beq
\vert d(e^fy_0)\vert_{e^{2f}y_0^2g}=e^{-f}\vert e^fdy_0\vert_{y_0^2g}+y_0\vert e^fdf\vert_{e^{2f}y_0^2g}=\vert dy_0\vert_{y_0^2g} \ \ \mbox{on} \ \partial \bar{M}.
\eeq
Note also that the non-trapping conditions is only about the behavior of geodesics on compact sets since as we will see later the behavior near infinity is already determined by hypothesis (1), (2) and (3).

\begin{proposition}\label{ligcon}
Let $(M,g)$ be an asymptotically de Sitter manifold that is  non-trapping. Then it is time-oriented and the future/past light cones are the subsets of $\Sigma_\pm\subset T^*M$ defined by
\beq
\Sigma_\pm=\{(x,\xi) \st  g_x(\xi,\xi)=0, \ {\rm the \ geodesic \ starting \ from\ } (x,\xi^\#) {\rm \ tends \ to\ } Y_\pm {\rm \ as \ } t\to +\infty\},
\eeq
where $^{\#}$ denotes the musical isomorphism. If the non-trapping at all energies is true we can define the timelike hyperboles $H_{\pm,\omega^2}\subset T^*M$  as the subsets
\beq
\{(x,\xi) \st g_x(\xi,\xi)=\omega^2>0, \  {\rm  the \ geodesic \ starting \ from\ } (x,\xi^\#) {\rm \ tends \ to\ } Y_\pm {\rm \ as \ } t\to +\infty\},
\eeq
and then 
\beq
{\rm Conv}(\Sigma_\pm)=\overline{\cup_{\omega\in\rr}H_{\pm,\omega^2}}.
\eeq
Where ${\rm Conv}(\Sigma_\pm)$ is the convex envelope of $\Sigma_\pm$.
\end{proposition}

\begin{proof}
First we need to understand the behavior of geodesics near $\p \bar{M}$. We know that geodesics can be viewed as integral curves $\gamma$ of the Hamiltonian vector field $H_{\vert \xi\vert_g^2}$ for $(x,\xi)\in T^*M$. In a neighborhood of $\p \bar{M}$ where $(y_0,y')$ is a coordinate system we have that $\frac{d}{dt}y_0\circ \gamma(t)=H_{\vert \xi\vert_g^2}y_0(\gamma(t)$. But $H_{\vert \xi\vert_g^2}y_0=\p_{\xi_0}{\vert \xi\vert_g^2}=-2g^{00}\xi_0$ where $(\xi_0,\xi')$ is the coordinate system on the cotangent space associated to $(y_0,y')$. From condition (3) we see that $g^{00}>0$ so
\beq
\frac{d}{dt}y_0\circ \gamma(t)=(-2g^{00}\xi_0)(\gamma(t))
\eeq
is of the opposite sign of $\xi_0$. $y_0^{-1}([0,\epsilon_0))$ is trivialy time oriented given the product decomposition \ref{prodec} so while in $y_0^{-1}([0,\epsilon_0))$ the sign of $\xi_0(\gamma)$ will not change. Therefore if $\gamma$ enters $y_0^{-1}([0,\epsilon_0))$ moving forward,   $y_0\circ \gamma(t)$ is decreasing so $\xi_0<0$ and it will stay this way until $\gamma$ leaves $y_0^{-1}([0,\epsilon_0))$, which is not possible since $y_0\circ \gamma(t)$ is decreasing.

To summarize we have found a neighborhood of the boundary that geodesics do not leave once they enter it. Let us define the following subset of $ T^*M$
\beq
\Sigma_\pm=\{(x,\xi)\vert \ g_x(\xi,\xi)=0 \vert {\rm \ the \ geodesic \ starting \ from\ } (x,\xi\#) {\rm \ tends \ to\ } Y_\pm {\rm \ in \ } t\to +\infty\}.
\eeq

We only need to show that locally $\Sigma_\pm$ are the two connected component of $g_x(\xi,\xi)=0$ with $\xi\neq 0$. Let $U$ be a compact connected neighborhood of $(x_1,\xi_1 )\in \Sigma_+$ on which we can use coordinates. Using those coordinates we have a decomposition of the cones in $\tilde \Sigma_\pm$ and we choose $(x_1,\xi_1)\in\tilde \Sigma_+$. Our goal is to show that all geodesics starting from $\tilde \Sigma_+^\#$ tend to $Y_+$ in $t\to+\infty$. By symmetry of the problem this will proves the proposition.

Let $\Phi: S^*U\times \rr_+\to M$ denote the geodesic on the unit sphere $S^*U\subset T^*M$ defined by taking the Riemannian metric on $U$ induced by the coordinate system. The flow is continuous so $(y_0\circ \Phi)^{-1}((0,\epsilon_0))$ is open and therefore is a neighborhood of all its points so there exists $t_{2,x}> t_{1,x}\geq 0$ and $V_x\subset S^*U$ open neighborhood of $x\in S^*U\cap\tilde \Sigma_+$ for all $x$ such that $(x,\rr_+)\cap (y_0\circ \Phi)^{-1}((0,\epsilon_0))\neq \emptyset$. However the non-trapping condition tells us that it is the case for all $x\in S^*U$ and since geodesics cannot leave $y_0^{-1}((0,\epsilon_0)$ we can take $t_{2,x}=+\infty$. $(V_x)_{x\in S^*U\cap\tilde \Sigma_+}=S^*U\cap\tilde \Sigma_+$ which is compact so it is still true for a finite number of $x_i$. Take $t_1=\max_i t_{1,x_i}$, we have
\beq
S^*U\cap\tilde \Sigma_+\times [t_1,\infty)\subset (y_0\circ \Phi)^{-1}((0,\epsilon_0))
\eeq
Because $Y_\pm$ are union of connected components that are compact we have that the indicator function $I$ of $(0,\epsilon_0)\times Y_+$ is continuous on $(y_0\circ \Phi)^{-1}((0,\epsilon_0))$. $S^*U\cap\tilde \Sigma_+\times (t_1,\infty)$ is connected and contains $(x,\xi_1,t_1+1)$, which has image $<\epsilon_0$ under $y_0\circ\Phi$. So $I\circ y_0\circ \Phi(S^*U\cap\tilde \Sigma_+\times (t_1,\infty))=\{1\}$. Which proves that all geodesics from $\tilde \Sigma_+$ tend to $Y_+$ in $t\to+\infty$.

A similar argument works to proves the statements at all energies.
\end{proof}
\begin{remark}
From this proof we can actually see that all future/past oriented geodesics tend to one single connected component of $Y_\pm$ in $t\to \infty$ on a connected component of $M$. So $M$ is just a union of connected asymptotically de Sitter manifolds with connected $Y_\pm$. 
\end{remark} 

\begin{definition}
Let $(M,g)$ be an asymptotically de Sitter manifold. A boundary defining function $y_0$ is called \emph{canonical} if
\beq
\vert dy_0\vert_{y_0^2g}=1 \ \mbox{in  a  neighborhood  of} \ \partial \bar{M}.
\eeq
For such a function $y_0$, a product structure $(y_0,\tilde y)$ is called \emph{canonical}, if the push-forward of the metric $g$ under $y_0^{-1}([0,\epsilon_0))\to [0,\epsilon_0)\times \partial \bar{M}, \ \ x\mapsto(y_0(x),\tilde y)$ is of the form
\beq
g=\frac{dy_0^2-g_0(y_0,\tilde y,d\tilde y)}{y_0^2},
\eeq
where $g_0(y_0,\tilde y,d\tilde y)$ is a family of Riemannian metrics on $\partial \bar{M}$ depending smoothly on $y_0\in[0,\epsilon)$.
\end{definition}

\begin{remark}
1) The boundary metric $(y_0^2g)\vert_{\partial \bar{M}}$ depends on the choice of the boundary defining function $y_0$. However, by Proposition \ref{bdfeq}, a different choice of $y_0$ multiplies the boundary metric by a conformal factor. Therefore, to each asymptotically de Sitter manifold corresponds a conformal class of Riemannian metrics on $\partial \bar{M}$:
\beq
[g]_{\partial \bar{M}}=\{(y_0^2g)\vert_{\partial \bar{M}}: y_0 \mbox{ is  a  boundary  defining  function}\}.
\eeq
2) The only property required of a canonical boundary defining function away from the boundary is that it must be smooth and non-zero therefore we can choose such functions according to the problem we want to solve. In our case we will assume that it is constant on some compact set of our choice. We will also ask that $\vert dy_0\vert_{y_0^2g}\leq1$ globally.

\end{remark}

\begin{theorem}
Let $(M,g)$ be an asymptotically de Sitter manifold and fix $g_0\in [g]_{\partial \bar{M}}$. Then there exists a canonical product structure $(y_0,y')$ on $\bar M$ such that
\beq
[g_0]_{\partial \bar{M}}=(y_0^2g)\vert_{\partial \bar{M}}.
\eeq
Any two such product structures coincide in a neighborhood of $\partial \bar{M}$. In other words, choosing a metric on the boundary is equivalent to defining a canonical product structure on a neighborhood of the boundary.
\end{theorem}

\subsection{Even conformal extension through the boundary}   We now need to define even asymptotically de Sitter manifolds since they are the only ones that can be extended as desired past the boundary.
\begin{definition}
An asymptotically de Sitter manifold $(M,g)$ is called \emph{even} if there exists a canonical product structure $(y_0,y')$ such that the corresponding metric $g_0$ satisfies
\beq
\partial^{2k+1}_{y_0}g_0(0,y',dy')=0, \ k\in \mathbb N_0.
\eeq
In other words, $g_0(y_0,y',dy')$ can be extended as an even function of $y_0$ past the boundary $\{y_0=0\}$.
\end{definition}

The analogue of this notion in Riemannian signature has first appeared in the work of \cite{Gui}.
\begin{theorem}
Assume that $(M,g)$ is an asymptotic de Sitter manifold with $(y_0,y')$ and $(\tilde y_0, \tilde y')$ two canonical products on $\bar M$, and $g_0$, $\tilde g_0$ the corresponding metrics. Assume that $\tilde g_0$ is even. Then $g_0$ is even and
\beq
\partial^{2k}_{\tilde y_0}(y_0)=0 \ on \ \partial \bar{M}, \ k\in\mathbb N_0,
\eeq
\beq
\partial^{2k+1}_{\tilde y_0}(\psi\circ y')=0 \ on \ \partial \bar{M}, \ k\in\mathbb N_0, \ \psi\in C^\infty(\partial \bar{M}).
\eeq

\end{theorem}

We now want to introduce the even extension of the de Sitter spaces beyond the boundary $\p \bar{M}$. For this we first use a remark from  \cite{VasyWaves} which says that under our assumptions $M$ must be globally hyperbolic and the boundary must split into two identical copies $Y_\pm$.

We will assume that we are able to find manifolds $M_\pm$ of dimension $n$ with given boundary $Y_\pm$. It is a classical topological result that this is always possible in dimension 4.

\begin{definition}\label{ext}
Let $(M,g)$ be an even asymptotically de Sitter manifold and fix a canonical product structure $(y_0,y')\in[0,\epsilon_0)\times \partial \bar{M}$. Consider the diffeomorphism
\beq 
M\cap\{y_0< \epsilon_0\}\to(0,\epsilon_0^2)\times\partial \bar{M}, \ x\mapsto(x_0,x'):=(y_0^2,y').
\eeq
We define:
\ben
\item The even compactification $\bar M_{\rm even}$ of $M$ to be the manifold with boundary obtained by gluing $M$ with $[0,\epsilon_0^2)\times\partial \bar{M}$ using the previous map.
\item The even compactification $\bar X_\epsilon$ of $M$ to be the manifold with boundary obtained by gluing $M_\pm$ with $[-\epsilon,\epsilon_0^2)\times Y_\pm$ using the previous map. 
\item We now glue $X_{2\epsilon}$ to  $M_\pm$ on the set $\{0> x_0> -2\epsilon \}$ to $\{0< x_0< \epsilon\}$ by the diffeomorphism $(x_0,x')\mapsto(-x_0,x')$. We now have a compact manifold $X$. We call $M':=X\setminus\bar M_{\rm even}$. We extend $x_0$ to the whole space $X$ by taking it constant far enough from $\p \bar{M}$. 
\item We say $X$ is \emph{non-trapping} if assumption (4) and (5) from Definition \ref{defman} hold true for $X$.
\een
\end{definition}

\begin{center}
 \includegraphics[scale=0.4]{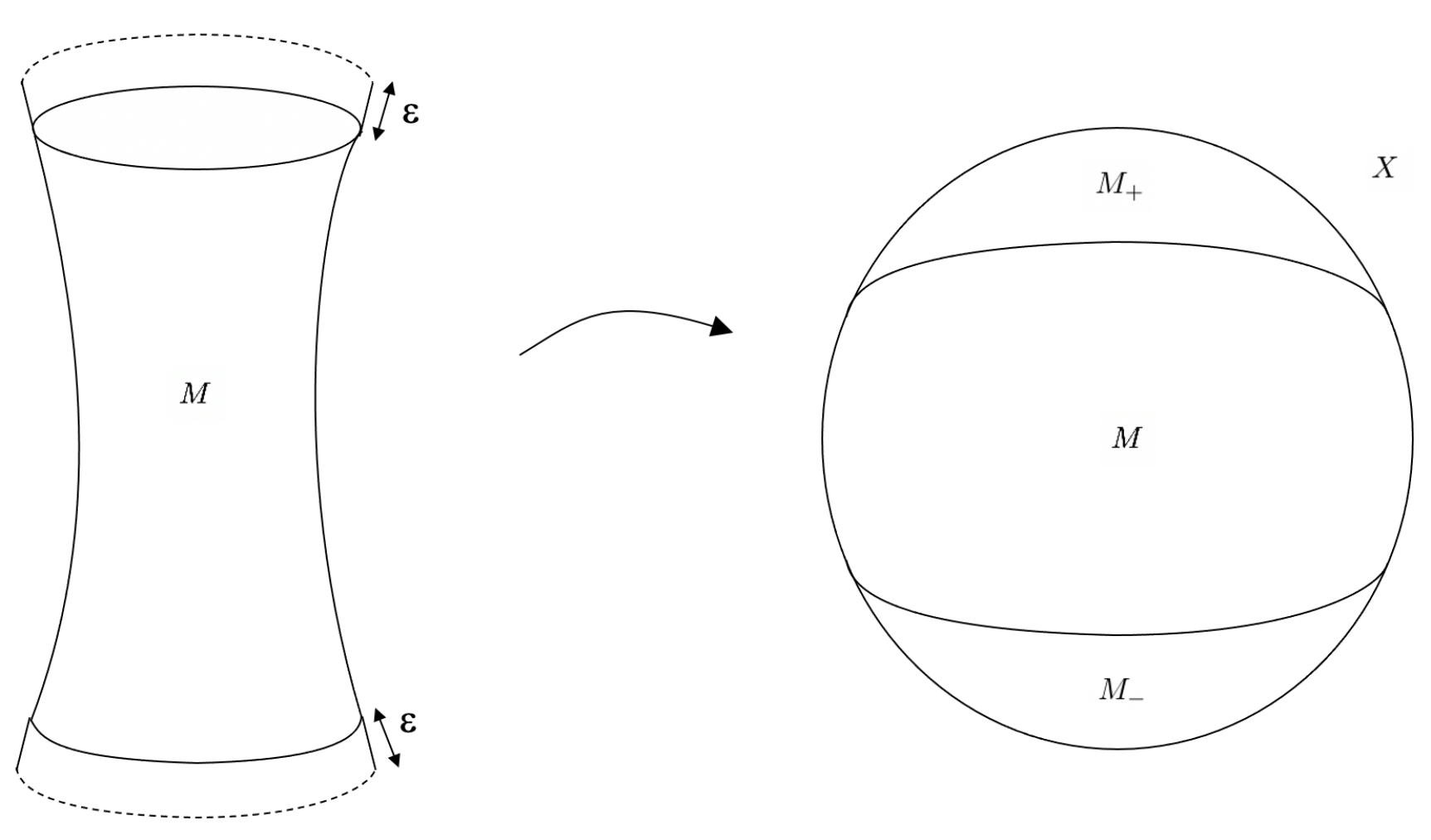}
 
 \bfseries Figure 1.  \normalfont  Geometric picture of the even extension of asymptotically de Sitter spaces. 
 \end{center}

\begin{remark}
1) For the exact de Sitter space we can choose to glue two copies of  hyperbolic space (which are non-trapping at all energies) so these assumptions hold true in the exact (non-perturbed) case.

2) We could also have glued $M$ onto itself to get rid of the topological assumption but then the global non-trapping assumption is not obvious even in the unperturbed case.
\end{remark}


\subsection{Conformally extended Laplace-Beltrami operator} 
Ler $(M,g)$ be an even asymptotically de Sitter manifold. Fix a canonical boundary defining function $y_0$ and let $(y_0,y')\in[0,\epsilon_0)\times\partial \bar{M}$ be the canonical product structure. Let $\bar M_{\rm even}$  be the even compactification of $M$ and let $X$ be its even extension.
 Put
\beq
x_0:=y_0^2
\eeq
so that $x_0$ is a boundary defining function of $\bar M_{\rm even}$ and $x_0+\epsilon$ is a boundary defining function of $\bar X_\epsilon$. We write
\beq
 M=\{x_0> 0\}, \ \bar Z:=\{-\epsilon\leq x_0< \epsilon_0^2\},
\eeq
where $\{y_0> \epsilon_0\}$ is the domain of the product structure $(y_0,y')$.

 On $\bar Z$, we have the product structure 
\beq
(x_0,x')\in[-\epsilon,\epsilon_0^2)\times M, \ x':=y'.
\eeq
Since $(M,g)$ is an even metric, we can write
\beq
g=\frac{dx_0^2}{4x_0^2}-\frac{g_0(x_0,x',dx')}{x_0},
\eeq
where $g_0$ is smooth in $x_0\in[0,\epsilon^2_0)$. We fix an extension of $g_0$ to $x_0\in [-\epsilon,\epsilon_0^2)$ as a smooth family of Riemannian metrics on $\partial \bar{M}$:
\beq
g_0\in C^\infty( [-\epsilon,\epsilon_0^2)\times \partial \bar{M}; \otimes^2_{\rm s}T^*\partial \bar{M}).
\eeq
In this way we can extend $g=\frac{dx_0^2}{4x_0^2}-\frac{g_0(x_0,x',dx')}{x_0}$ to $X_\epsilon\setminus\p \bar{M}$.  Using a partition of unity we can extend $g$ to $M_\pm$ by taking it to be $\frac{dx_0^2}{4x_0^2}-\frac{g_0(x_0,x',dx')}{x_0}$ on a neighborhood of $\p \bar{M}$ and some Riemannian metric outside $X_\epsilon$.

\begin{center}
	\includegraphics[scale=0.5]{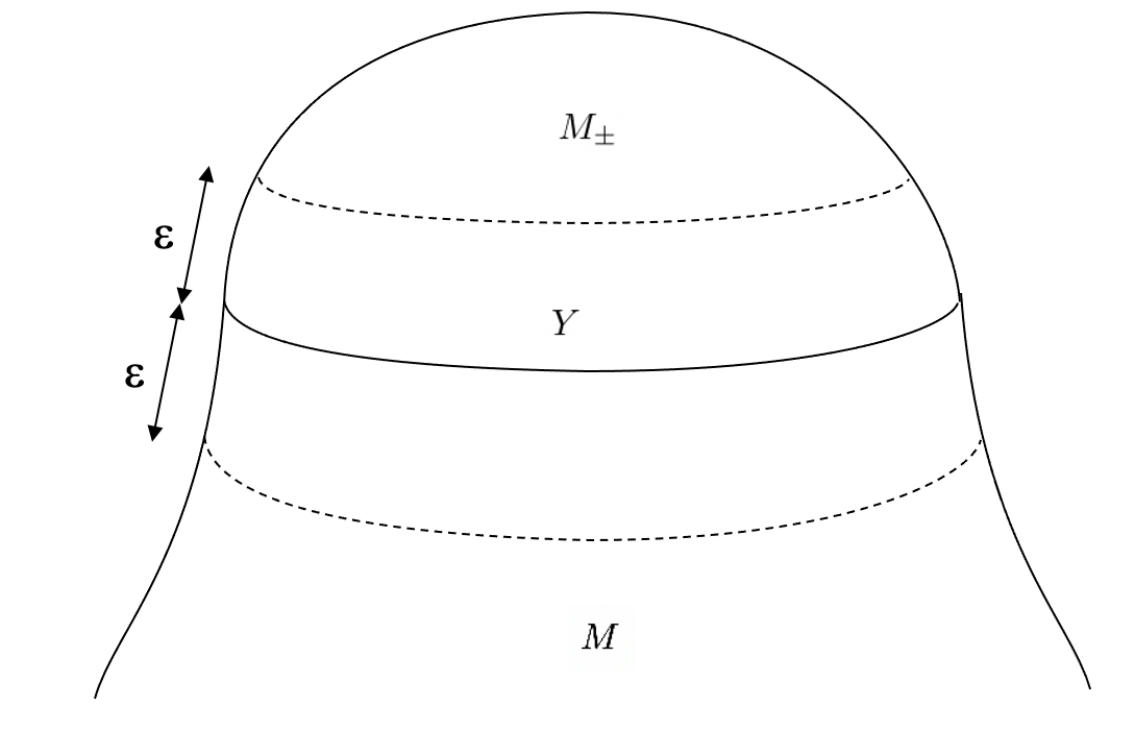}
	
	\bfseries Figure 2. \normalfont  Geometric picture of the even extension of asymptotically de Sitter spaces near the boundary. 
\end{center}

Our goal is to modify the Laplace-Beltrami operator so that it can be defined on  conformally extended asymptotically de Sitter spaces. To do that we note that according to \cite{VasyWaves} and \S5.2 of \cite{DyaZwo}, the eigenfunctions associated to the eigenvalue $\lambda$ of the Laplace-Beltrami operator on $M$ and $M'$ are of the form 
$$
u\in \vert x_0\vert^{\frac{n-1}{4}\pm \frac{1}{2}\sqrt{\frac{(n-1)^2}{4}+\lambda}}C^\infty(X).
$$
This motivates us  to rescale the Laplace--Beltrami operator to obtain accommodate for this behavior and also multiply it by $x_0$ to make the coefficients bounded, and  take $-\lambda^2-\frac{(n-1)^2}{4}$ to be the spectral parameter instead of $\lambda$ to simplify the formulas. 

\begin{proposition} Let us define the operator $P(\lambda)$ as
\beq
P(\lambda):=x_0^{\frac{i\lambda}{2}-\frac{n+3}{4}}\left(\square_g+\lambda^2+\frac{(n-1)^2}{4}\right)x_0^{\frac{n-1}{4}-\frac{i\lambda}{2}} \mbox{ on}\ M,
\eeq
and
\beq
\bea
P(\lambda)&:=4x_0\partial_{x_0}^2+4(i\lambda-1)\partial_{x_0}-\gamma (4x_0\partial_{x_0}+n-1-2i\lambda)-\Delta_{g_0} \mbox{ on} \ \bar Z, \\
P(\lambda)&:=-(-x_0)^{\frac{i\lambda}{2}-\frac{n+3}{4}}\left(\Delta_g+\lambda^2+\frac{(n-1)^2}{4}\right)(-x_0)^{\frac{n-1}{4}-\frac{i\lambda}{2}} \mbox{ on} \ M_\pm,
\eea
\eeq
where $\square_{g}$, $\Delta_g$ and $\Delta_{g_0}$ are the Laplace-Beltrami operator for $g$ and $g_0$ respectively and
\beq
\gamma(x_0,x'):=J^{-1}\frac{\partial J}{\partial x_0}\in C^\infty(\bar Y;\mathbb R), \ J:=\sqrt{\vert \det g_0\vert}.
\eeq
Then, those definitions are consistent on $M\cap \bar Z$ and $M_\pm\cap \bar Z$.
\end{proposition}
\begin{proof}
We recall $g=\frac{dx_0^2}{4x_0^2}-\frac{g_0(x_0,x',dx')}{x_0}$, we compute
\beq
\bea
\square_g &=\frac{1}{\sqrt{\vert \det g\vert}}\partial_i(\sqrt{\vert \det g\vert}g^{ij}\partial_j)\\
 &=\frac{1}{2x_0^{-\frac{n+1}{2}}\sqrt{\vert \det g_0\vert}} \, \left( \partial_{x_0}(2x_0^{-\frac{n+1}{2}}\sqrt{\vert \det g_0\vert}4x_0^2\partial_{x_0})-2x_0^{-\frac{n-3}{2}}\partial_{x'_i}\sqrt{\vert \det g_0\vert}g_0^{ij}\partial_{x'_j}\right)\\
&=\frac{x_0^{\frac{n+1}{2}}}{\sqrt{\vert \det g_0\vert}} \partial_{x_0}(\sqrt{\vert \det g_0\vert}4x_0^{\frac{3-n}{2}}\partial_{x_0})-x_0^{}\Delta_{g_0}\\
&=4x_0^{\frac{n+1}{2}}\partial_{x_0}x_0^{\frac{3-n}{2}}\p_{x_0}+4x_0^2\gamma\partial_{x_0}-x_0^{}\Delta_{g_0}
\eea
\eeq
Now we compute
\beq
\bea
&x_0^{\frac{i\lambda}{2}-\frac{n+3}{4}}(\square_g+\lambda^2+\frac{(n-1)^2}{4})x_0^{\frac{n-1}{4}-\frac{i\lambda}{2}}\\
&=x_0^{\frac{i\lambda}{2}-\frac{n+3}{4}}(4x_0^{\frac{n+1}{2}}\partial_{x_0}x_0^{\frac{3-n}{2}}\p_{x_0}+4x_0^2\gamma\partial_{x_0}) x_0^{\frac{n-1}{4}+\frac{i\lambda}{2}}-\frac{1}{x_0}(\lambda^2+\frac{(n-1)^2}{4})-\Delta_{g_0}
\eea
\eeq
Using the identity $x_0\p_{x_0}x_0^\alpha=x_0^\alpha(x_0\p_{x_0}+\alpha)$, we compute
\beq
\bea
&=\frac{1}{x_0}(2x_0\p_{x_0}+\frac{1-n}{2}-i\lambda)(2x_0\p_{x_0}+\frac{n-1}{2}-i\lambda)-\gamma(4x_0\p_{x_0}+n-1-2i\lambda)-\Delta_{g_0}+\frac{1}{x_0}(\lambda^2+\frac{(n-1)^2}{4})\\
&=4x_0\partial_{x_0}^2+4(i\lambda-1)\partial_{x_0}-\gamma (4x_0\partial_{x_0}+n-1-2i\lambda)-\Delta_{g_0},
\eea
\eeq
hence the desired conclusion for $M\cap \bar Z$. A similar computation works for $(M_+\cup M_-)\cap \bar Z$. 
\end{proof}

\begin{proposition}
We define the volume form
\beq
\label{volform}
d{\rm Vol}= \begin{cases}
        2x_0^{\frac{n+1}{2}}d{\rm Vol}_g & on \ M\cup M_+\cup M_-,\\
        dx_0 \,d{\rm Vol}_{g_0} & on \ Z.
    \end{cases}
\eeq
For this volume form we have
\beq\label{adjoint}
P(\lambda)^*=P(\bar \lambda)
\eeq
in the sense of formal adjoints with respect to the Hermitian product on $L^2(X,d\rm{Vol})$.
\end{proposition}

\begin{proof}
We will show that $P(\lambda)^*=P(\bar \lambda)$ on $M\cup M_+\cup M_-$. Let $u,v\in C^\infty_{\rm c}(M\cup M_+\cup M_-)$. Then
\beq
\bea
\langle v,P(\lambda)u\rangle &=2\int \bar v \left(x_0^{\frac{i\lambda}{2}-\frac{n+3}{4}}\big(\square_g+\lambda^2+\frac{(n-1)^2}{4}\big)x_0^{\frac{n-1}{4}-\frac{i\lambda}{2}} u\right) x_0^{\frac{n+1}{2}}d{\rm Vol}_g\\
&=2\int \overline{x_0^{\frac{n-1}{4}-\frac{i\lambda}{2}} v }\big(\square_g+\lambda^2+\frac{(n-1)^2}{4}\big)x_0^{\frac{n-1}{4}-\frac{i\lambda}{2}} u \,d{\rm Vol}_g,
\eea
\eeq
and since the Laplace--Beltrami operator is symmetric we get the result.
\end{proof}

In this paper we will study the high frequency behavior of $P(\lambda)$ for large $\lambda$. To do this we will use semi-classical analysis and to that end we define the semi-classical operator
\beq
P_h(\omega):=h^2 P\left(\frac{\omega}{h}\right)
\eeq
for $h>0$, so in coordinates we have 
\beq
P_h(\omega)=x_0^{\frac{i\omega}{2h}-\frac{n+3}{4}}\left(h^2\square_g+\omega^2+\frac{(n-1)^2}{4}\right)x_0^{\frac{n-1}{4}-\frac{i\omega}{2h}}, \mbox{ on}\ M\
\eeq
and 
\beq
P_h(\omega)=4x_0(h\partial_{x_0})^2+4(i\omega-h)(h\partial_{x_0})-\gamma \left(4hx_0(h\partial_{x_0})+h^2(n-1)-2ih\omega\right)-h^2\Delta_{g_0}, \mbox{ on} \ \bar Y, \eeq
on $\bar Y$ , and
\beq 
P_h(\omega)=-(-x_0)^{\frac{i\omega}{2h}-\frac{n+3}{4}}\left(h^2\Delta_g+\omega^2+\frac{(n-1)^2}{4}\right)(-x_0)^{\frac{n-1}{4}-\frac{i\omega}{2h}}, \mbox{ on} \ M_\pm.
\eeq

\begin{proposition}\label{symbsem}
Let $\omega$ be a smooth bounded function in $h\in (0,h_0]$. Then $P_h(\omega)$ is a semi-classical differential operator of order $2$. For $\omega=\omega_R+h\omega_I$ with $\omega_{R/I}\in\rr$ we have that $\Im P_h(\omega):=\frac{P_h(\omega)-P_h(\omega)^*}{2i}\in h\Diff^1_h(X)$ and
\beq
\sigma_h(h^{-1}\Im P_h(\omega))=\omega_I\p_{\omega_R} p_{\omega}(x,\xi)
\eeq
where $p_\omega(x,\xi)$ is the semi-classical symbol of $P_h(\omega)$ and is real and independent of $\omega_I$.
\end{proposition}
\begin{proof} On $\bar Z$ it is clear that $P_h(\omega)\in \Diff^2_h$ because $P_h(\omega)=4x_0(h\partial_{x_0})^2+4(i\omega-h)(h\partial_{x_0})-\gamma \left(4hx_0(h\partial_{x_0})+h^2(n-1)-2ih\omega\right)-h^2\Delta_{g_0}$. On $M$ we have that 
\beq
P_h(\omega)=x_0^{-1}\left(h^2\square_g+\omega^2+\frac{(n-1)^2}{4}\right)+x_0^{\frac{i\omega}{2h}-\frac{n+3}{4}}h^2[\square_g,x_0^{\frac{n-1}{4}-\frac{i\omega}{2h}}].
\eeq
Observe that $\square_g$ is a sum of operators of the form $a\p_1\p_2+b\p_3+c$ where $a,b,c$ are smooth coefficients and $\p_i$ represent a partial derivative in some direction. Moreover, we have that for general operators $[AB,C]=[A,[B,C]]+[B,C]A+[A,C]B$. So for $\alpha\in \cc$, $[\square_g,x_0^\alpha]=a[\p_1,[\p_2,x_0^\alpha]]+a[\p_1,x_0^\alpha]\p_2+a[\p_2,x_0^\alpha]\p_1+b[\p_3,x^\alpha]$ and $[\p_i,x_0^\alpha]=\alpha\p_ix_0x_0^{\alpha-1}$, thus
\beq
x_0^{\frac{i\omega}{2h}-\frac{n+3}{4}}[h^2\square_g,x_0^{\frac{n-1}{4}-\frac{i\omega}{2h}}]=\left(\frac{n-1}{4}h-\frac{i\omega}{2}-h\right)\left(\frac{n-1}{4}h-\frac{i\omega}{2}\right)F+\left(\frac{n-1}{4}h-\frac{i\omega}{2}\right)hG,
\eeq
where $F$ is a real smooth function, $G$ is a real first order differential operator,  and both independent of $h$ and $\omega$. Therefore $P_h(\omega)\in \Diff^2_h(X)$ and
\beq\label{symbexpl}
p_\omega(x,\xi)=x_0^{-1}\left(-\vert \xi\vert_g^2+\omega_R^2(1-\frac{ F}{4})-\frac{\omega_R}{2}\sigma_h(ihG)\right)\in\rr.
\eeq
Further computations yield
\beq
F=\frac{\vert dx_0\vert^2_g}{x_0^2}, \quad G=(2g^{ij}(\p_ix_0)\p_j+\frac{1}{\sqrt{\vert \det g\vert}}\p_i(\sqrt{\vert \det g\vert}g^{ij}(\p_j x_0)))x_0^{-1},
\eeq
and therefore
\beq\label{symexa}
p_\omega(x,\xi)=x_0^{-1}\left(-\vert \xi\vert_g^2+\omega_R^2(1-\frac{\vert dx_0\vert^2_g}{4x_0^2})+\frac{\omega_R}{x_0}g(dx_0,\xi)\right).
\eeq
We then notice that $P_h(\omega_R)^*=P_h(\omega_R)$ according to \eqref{adjoint} so $\Im P_h(\omega)=\Im(P_h(\omega)-P_h(\omega_R))$. Using our previous notation we have that $$
\bea P_h(\omega)-P_h(\omega_R)&=\omega^2-\omega_R^2+\left(\frac{n-5}{4}h-\frac{i\omega}{2}\right)\left(\frac{n-1}{4}h-\frac{i\omega}{2}\right)F \fantom -\left(\frac{n-1}{4}h-\frac{i\omega_R}{2}-h\right)\left(\frac{n-1}{4}h-\frac{i\omega_R}{2}\right)F+\frac{\omega_I}{2}hG.
\eea
$$ 
We see that it is in $h\Diff^1_h$ and that its image in $h\Diff^1_h/h^2\Diff^1_h$ is linear in $i\omega_Ih$ and purely imaginary, therefore it is equal to $\omega_Ih\p_{\omega_I}P_h(\omega)=i\omega_Ih\p_{\omega_R} P_h(\omega_R)$ so $\sigma_h(h^{-1}\Im P_h(\omega))=\omega_I\p_{\omega_R}\sigma_h(P_h(\omega_R)=\omega_I \p_{\omega_R}p_\omega$ since $P_h(\omega)-P_h(\omega_R)\in h\Diff^1_h$. 
\end{proof}

\section{Distinguished inverses}\label{DisInv}

In this section we give an overview of the construction of the four microlocally distinguished inverses. The proof follows the work of \cite{DyaZwo} and \cite{VasyEst}. A very similar work has been done in \cite{VasBasWun} and similar estimates have been recovered. The differences with \cite{VasBasWun} come from the setting which is a little bit more general since we consider asymptotically de Sitter-like spaces and we construct the Feynman and anti-Feynman propagator as well as the forward and backward propagator. All in all these differences in the proof are minor.

\subsection{Phase space behavior}

Let 
$$p=-4x_0\xi_0^2+\vert\xi'\vert_{g_0}^2\ \mbox{ near  } \p \bar{M}
$$
be the principal symbol of $P(\lambda)$. We will often use the coordinates $(\rho, \hat\xi)$ on the fibers of $\bar T^*Y\setminus 0$ defined as
\beq
\rho=(\xi_0^2+\vert\xi'\vert_{g_0}^2)^{-1/2}\in[0,\infty), \ \hat\xi=(\hat\xi_0,\hat\xi')=\rho\xi\in S^*_xY,
\eeq
where $S^*Y$ denote the cosphere bundle
\beq
S^*Y=\{(x,\hat \xi_0,\hat \xi') \st \hat\xi_0^2+\vert\xi'\vert_{g_0}^2=1\}.
\eeq

Using those coordinate we define the \emph{radial sets}:
\beq
L_\pm:=\{\rho=0,  \ x_0=0, \ \hat\xi'=0, \ \hat\xi_0=\pm 1\}.
\eeq
\begin{proposition}
	\label{radialsets}
	According to definition \ref{defradsets}, $L_-$ is a radial source and $L_+$ is a radial sink for the second order differential operator $P(\lambda)$ in the sense that:
	\beq
	H_p\rho\vert_{L_\pm}=\mp \beta_0, \beta_0\in C^\infty(L_\pm), \ \beta_0>0
	\eeq
	There exist a non-negative homogeneous of order $0$ quadratic defining function $\rho_0$  (i.e. it vanishes quadratically on $L_\pm$ and is non-degenerate) and $\beta_1>0$ on a neighborhood of $L_\pm$  such that
	\beq
	\mp\rho H_p\rho_0-\beta_1\rho_0
	\eeq
	is non-negative modulo cubically vanishing terms in $L_\pm$.
	
	We also have that the principal symbol $p_1$ of $\Im P(\lambda)$ is real and
	\beq
	p_1\vert_{L\pm}:=\mp \tilde\beta\beta_0  \rho^{-1}, \ \tilde\beta\in C^\infty(L_\pm).
	\eeq
	Here $\tilde \beta=-\Im \lambda$ on $L_\pm$
\end{proposition}

This proposition tells us that we will be able to use \ref{estimeesradiales} to gain regularity from the high regularity estimates. In particular we have regularity along the bicharacteristics. It is therefore essential to understand the Hamiltonian flow.

In order to understand the global behavior of bicharacteristics we define 
$$
\Lambda^{\epsilon_1}_{\epsilon_2}:=\{x\in Y_{\epsilon_1}\}\cap L_{-\epsilon_1\epsilon_2}
$$ 
with $\epsilon_i=\pm$. Here $\epsilon_1$ refers to the position of the radial set at ``past'' or ``future'' infinity and $\epsilon_2$ refers to being in the ``past oriented'' and ``future oriented'' light cones. Indeed we could also have defined them as $\Lambda^{\epsilon_1}_{\epsilon_2}=\{x\in Y_{\epsilon_1}\}\cap \hat \Sigma_{\epsilon_2}$. We write $\Lambda_\pm=\Lambda_\pm^+\cup \Lambda_\pm^-$ and $\Lambda^\pm=\Lambda^\pm_+\cup \Lambda^\pm_-$.

\begin{proposition}
	\label{globdyn}
	For $\epsilon=\pm$ and $\forall (x,\xi)\in \p\bar T^*\bar X$ causal (i.e. $\rho^2p(x,\xi)\leq 0$) and future oriented, we have $e^{\pm tH_p}(x,\epsilon\xi)\to \Lambda^{\mp \epsilon}_\epsilon$ as $t\to+\infty$. In other words on each light cone, bicharacteristics go from sources to sinks.
\end{proposition}

Now that the microlocal phase space behavior is well understood we will need to understand the semi-classical phase space behavior. And although we have the same sources and sinks the principal symbol is different. Indeed to use semi-classical estimates we need to understand the dynamics of the bicharacteristics of the semi-classical symbol $p_\omega$ of $h^2P(h^{-1}\omega_R+i\omega_I)$. Near $\p \bar{M}$, $p_{\omega}=-4x_0\xi_0^2+\vert\xi'\vert_{g_0}^2+4\omega_R\xi_0$.

\begin{theorem}{\rm\emph{(Global dynamics.)}}\label{glodyn}
	Let $\epsilon=\pm$. For all $\omega_R\neq0$  and $(x,\xi)\in\Sigma_{\epsilon,\omega}\subset \bar T^*X$, $\gamma(t_0):= e^{t\rho H_{p_\omega}}(x,\xi)$, the following holds:
	
	If $(x,\xi)\in \Sigma_{\epsilon,\omega}\setminus L_+\cup L_-$ and $\omega_R> 0$,
	
	\quad (a) $\gamma(t)\to \Lambda^{\epsilon}_\epsilon$  as $ t\to -\infty$ without leaving $\bar M$,
	
	\quad (b) $\gamma(t)\to\Lambda^{-\epsilon}_\epsilon$   as $ t\to +\infty$. 
	
	If $(x,\xi)\in \Sigma_{\epsilon,\omega}\setminus L_+\cup L_-$ and $\omega_R< 0$,
	
	\quad (a) $\gamma(t)\to \Lambda^{\epsilon}_\epsilon$ as $ t\to -\infty$,
	
	\quad (b) $\gamma(t)\to\Lambda^{-\epsilon}_\epsilon$ as $ t\to +\infty$ without leaving $\bar M$.
	
	In other words null-bicharacteristics of $p_{\omega}$ go from sources of $P(\lambda)$ ($\Lambda^\epsilon_\epsilon\subset L_-$) to sinks ($\Lambda^{-\epsilon}_\epsilon\subset L_+$) but can escape either at the sources or the sinks depending on the sign of $\omega_R$ but because of the global non-trapping hypothesis they will reach the radial set eventually.
\end{theorem}

\subsection{Fredholm theory of $P(\lambda)$}
In this section we want to investigate the mapping properties of the transformed wave operator.

In the following we will use the volume form $d{\rm Vol}$ on $X$ to define  Sobolev spaces. However, since $X$ is compact, all volume forms yields equivalent norms. We will also use variable order Sobolev spaces which we define in Definitions \ref{varsob} and \ref{semsob}.

Let us introduce the threshold value $S(\lambda)=\frac{1}{2}-\Im\lambda$. This is the value that appears in Section \ref{estimeesradiales} for $P(\lambda)$. 
In general this value could be different at the sources and sinks and it would not cause any problem. We will denote $S^*(\lambda)=\frac{1}{2}+\Im\lambda=1-S(\lambda)$ the threshold value for $P(\lambda)^*$. 

Let us recall $
\Lambda^{\epsilon_1}_{\epsilon_2}:=\{x\in Y_{\epsilon_1}\}\cap L_{-\epsilon_1\epsilon_2}
$ and we write $\Lambda_\pm=\Lambda_\pm^+\cup \Lambda_\pm^-$ and $\Lambda^\pm=\Lambda^\pm_+\cup \Lambda^\pm_-$.

\begin{definition}\label{varord}
	Let $s_{\rm ftr}\in S^0(X)$ be such that:
	\ben
	\item $s_{\rm ftr}$ is constant near the radial sets $\Lambda:=\Lambda_+\cup\Lambda_-=\Lambda^+\cup\Lambda^-$,
	\item $s_{\rm ftr}$ is non-increasing along the $H_p$ flow on $\Sigma_{+,\omega}\cap\{x_0\geq -\epsilon_0^2\}$ and non-decreasing
	on $\Sigma_{-,\omega}\cap\{x_0\geq -\epsilon_0^2\}$,
	\item $s_{\rm ftr}\vert_{\Lambda^+}<S(\lambda)<s_{\rm ftr}\vert_{\Lambda^-}$.
	\een
	Let $s_{\rm past}\in S^0(X)$ be such that it satisfies the same conditions as $s_{\rm ftr}$ but with $\Lambda^+$ and $\Lambda^-$ reversed as well as ``increasing'' with ``decreasing''. In other words $2S(\lambda)-s_{\rm past}$ satisfies the same conditions as $s_{\rm ftr}$.

	
	
	
	Let $s_{\rm F}\in S^0(X)$ be such that
	
	\quad (1) $s_{\rm F}$ is constant near the radial sets $\Lambda$.
	
	\quad (2) $s_{\rm F}$ is non-decreasing along the $H_p$ flow on $(\Sigma_{+,\omega}\cup\Sigma_{-,\omega}) \cap\{x_0\geq -\epsilon_0^2\}$.
	
	\quad (3) $s_{\rm F}\vert_{L^+}<S(\lambda)<s_{\rm F}\vert_{L_-}$
	
	Let $s_{\rm \bar F}\in S^0(X)$ be such that it satisfies the same conditions as $s_{\rm ftr}$ be with $L^+$ and $L^-$ reversed as well as ``increasing'' and ``decreasing''. In other words $2S(\lambda)-s_{\rm\bar F}$ satisfy the same conditions as $s_{\rm F}$. Here $\rm F$ and $\rm\bar F$ refer to \emph{Feynman} and \emph{anti-Feynman}.
	Let $s_{\rm ftr/past/F/\bar F}^*\in S^0(X)$ be such that it satisfies the same conditions as $s_{\rm ftr/past/F/\bar F}$ but with $S(\lambda)$ replaced by $S^*(\lambda)$,  $L^+/L^-$ or $\Lambda^+/\Lambda^-$ reversed as well as ``increasing'' and ``decreasing''. In other words $1-s_{\rm ftr/past/F/\bar F}^*$ satisfies the same conditions as $s_{\rm ftr/past/F/\bar F}$.
	
	
	
	
	
	
	
	
	
\end{definition}

The goal here is to use the radial estimates from sections \ref{Mic} and \ref{Sem} but to use the variable order to be able to use both high and low regularity estimates. The condition on the order and the flow of the principal symbol is here so the we can glue those estimates using propagation of singularity and make sure we control the whole space.

\begin{theorem}\label{globest}

	If $u\in \mathcal D'(X)$ is such that for $S(\lambda)<m<(s_{\rm  ftr/past/F/\bar F})\vert_{\Lambda^-/\Lambda^+/L^+/L^-}$, $\WF^m(u)\cap \Lambda^-/\Lambda^+/L^-/L^+=\emptyset$, then for all $N\in\rr$ 
	\beq
	\norm{u}_{s_{\rm ftr/past/F/\bar F}}\leq C(\norm{P(\lambda)u}_{s_{\rm ftr/past/F/\bar F}-1}+\norm{u}_{-N}).
	\eeq
	
	If $u\in \mathcal D'(X)$ is such that for $S(\lambda)<m<s_{(\rm ftr/past/F/\bar F)}^*\vert_{\Lambda^+/\Lambda^-/L^+/L^+}$, $\WF^m(u)\cap \Lambda^+/\Lambda^-/L^+/L^-=\emptyset$, then for all $N\in\rr$  
	\beq
	\norm{u}_{s_{\rm ftr/past/F/\bar F}^*}\leq C(\norm{P(\lambda)^*u}_{s_{\rm ftr/past/F/\bar F}^*-1}+\norm{u}_{-N}). 
	\eeq
	
\end{theorem}
The same thing can be done semi-classically to obtain similar estimates with an error term decaying in $h$ and that can therefore be absorb on the left hand side for small $h$, i.e. high $\Re\lambda$.

\begin{theorem}\label{invest}
	For $\module{\Re \lambda} \neq 0$ big enough 
	and $s_{\rm F/\bar F}$, $s_{\rm F/\bar F}^*$ as in Definition \ref{varord}, then $\exists C>0$ independent of $\Re\lambda$ such that for all $u\in \mathcal D'(X)$ with  $\WF^m(u)\cap L^+/L^-=\emptyset$  for any $S(\lambda)<m<s_{\rm F/\bar F}\vert_{L^+/L^-}$,
	\beq\label{invest1}
	\norm{u}_{s_{\rm F/\bar F}}\leq C\module{\Re\lambda}^{-1}\norm{P(\lambda)u}_{s_{\rm F/\bar F}-1}
	\eeq
	and 
	\beq\label{invest2}
	\norm{u}_{s_{\rm F/\bar F}^*}\leq C\module{\Re\lambda}^{-1}\norm{P(\lambda)^*u}_{s_{\rm F/\bar F}^*-1}.
	\eeq
\end{theorem}
Both theorems together show that $P(\lambda)$ is Fredholm and invertible for big $\vert \Re\lambda\vert$. More precisely:
\begin{definition}
	Let us define 
	\beq
	\pazocal X^s:=\{u\in H^s(X) \st \norm{P(0)u}^2_{H^{s-1}(X)}<+\infty\}
	\eeq
	endowed with the norm
	\beq
	\norm{u}_{\pazocal X^s}=( \norm{u}^2_{H^{s}(X)}+ \norm{P(0)u}^2_{H^{s-1}(X)})^{1/2}.
	\eeq
\end{definition}

This is a Hilbert space since it can be identified with the closed subspace
\beq
\{(u,f)\st  P(0)u=f\}\subset H^s(X)\oplus H^{s-1}(X),
\eeq
and since $P(\lambda)-P(0)$ is a first order operator we have that
\beq
P(\lambda):\pazocal X^s\to H^{s-1}(X)
\eeq
is a bounded holomorphic family of operators.

\begin{theorem}\label{meroinv}
	Fix $s_{\rm F/\bar F}\in S^0(X)$ as in Definition \ref{varord}. Then for $\inf s_{\rm F/\bar F}<S(\lambda)=\frac{1}{2}-\Im\lambda< \sup s_{\rm F/\bar F}$, $P(\lambda):\pazocal X^{s_{\rm F/\bar F}}\to H^{s_{\rm F/\bar F}-1}(X)$ is a Fredholm operator of index $0$  and it has a meromorphic inverse with poles of finite rank,
	\beq
	P^{(s_{\rm F/\bar F})}(\lambda)^{-1}:H^{s_{\rm F/\bar F}-1}(X)\to \pazocal X^{s_{\rm F/\bar F}},
	\eeq
	and if $s_{\rm F/\bar F}\leq s'_{\rm F/\bar F}$ then the poles of $P^{(s'_{\rm F/\bar F})}(\lambda)^{-1}$ such that $\inf s_{\rm F/\bar F}<S(\lambda)=\frac{1}{2}-\Im\lambda<\sup \ s_{\rm F/\bar F}$  are poles of $P^{(s_{\rm F/\bar F})}(\lambda)^{-1}$ and for $\Im\lambda$ in this interval  
	$$
	P^{(s'_{\rm F/\bar F})}(\lambda)^{-1}=P^{(s_{\rm F/\bar F})}(\lambda)^{-1}\vert_{H^{s'_{\rm F/\bar F}-1}(X)}.
	$$  In particular, let $a<b$ and 
	$$
	O_{[a,b]}:=\{s_{\rm F/\bar F}  \ \mbox{as \ in \ Definition \ \ref{varord}} \st \inf s_{\rm F/\bar F}<a, \ b< \sup s_{\rm F/\bar F}\}.
	$$
	Then
	\beq
	P(\lambda)^{-1}:C^\infty(X)\to \bigcap_{s_{\rm F/\bar F}\in O_{\{S(\lambda)\}}} \pazocal X^{s_{\rm F/\bar F}}
	\eeq
	is uniquely defined for $\lambda$ in $\cc$ outside a set of isolated points as the restriction of $P^{(s_{\rm F/\bar F})}(\lambda)^{-1}$ to ${C^\infty(X)}$ for $\lambda\in \rr+i[a,b]$ as long as $s_{\rm F/\bar F}\in O_{[a,b]}$.
\end{theorem}
The proof can be found in \cite{DyaZwo} Theorem 5.30.
\subsection{Uniform wavefront sets}

In this section we follow Dang--Wrochna \cite{DanWro}. There are several ways to define the wavefront set of an operator. One is to define it as the wavefront set of its Schwartz kernel, the second is called \emph{operatorial wavefront set} and can be adapted to include uniformity aspects. It is defined as follows:
\begin{definition}
	We define the wavefront set $\WF'(A)\subset  S^* M  \times S^* M$ of an operator $A\in \bigcup_{m\in\rr}\bigcap_{s_0\in \rr}B\big(H^{s_0}_{\rm c}(M),H^{s_0+m}_{\rm loc}(M)\big)$ by saying that an pair of points $(q_1,q_2)\in S^* M  \times S^* M$ are not in the wavefront set of $A$ if and only if there exists $B_1,B_1\in \Psi^0(M)$ elliptic at $q_1$ and $q_2$ respectively such that
	\beq
	B_1 A B_2 \in B(H^{s_0}_\c(M), H_\loc^{s_1}(M))
	\eeq
	for all $s_0,s_1\in \rr$.
\end{definition}

We will be interested in operators depending on a complex parameter and of controlling $O(\bra z\ket^{-1/2})$ decay for large $\vert z\vert $. The uniform wavefront set is defined as:
\begin{definition}\label{defrrr} Let $Z\subset \cc$ and suppose $\{ G(z)\}_{z\in Z}$  is for all $m\in\rr$ a bounded family of operators in $B(H^m_\c(M),H^m_\loc(M))$.
	The \emph{uniform operator wavefront set of order $s\in\rr$ and weight $\bra z\ket^{-\12}$} of $\{ G(z)\}_{z\in Z}$ is the set 
	\beq\label{eq:wfs}
	\wfl{12}\big( G(z) \big)\subset (T^*M\setminus\zero)\times (T^*M\setminus\zero)
	\eeq
	defined as follows:  $((x_1;\xi_1),(x_2;\xi_2))$ is \emph{not} in \eqref{eq:wfs} if and only if there exists a uniformly bounded  family $B_i(z)\in \Psi^{0}(M)$ of  properly supported operators, each elliptic at $(x_i;\xi_i)$ and such that for all $r\in \rr$, the family
	$$
	\bra z\ket^{\12}B_1(z) G(z) B_2(z)^* \mbox{ for } z\in Z \mbox{ is bounded in } B(H^{r}_\c(M), H_\loc^{r+s}(M)).
	$$
\end{definition}

We define the \emph{uniform operator wavefront set of order $s\in\rr$ and weight $1$} in the same way,  with $\bra z\ket^{\12}$ replaced by $1$, and we denote that set $\wfl{0}\big( G(z) \big)$ for simplicity.  Definition \ref{defrrr} is  similar to the definition from  \cite[\S3]{DanWro}, with the only difference  that we allow $B_i$ to depend on $z$ (which is easier to verify in practice).


Let us denote by $\Delta^*$ be the diagonal in $(T^*M \setminus \zero)^{\times 2}$, i.e.
$$
\Delta^* = \{ ((x_1;\xi_1),(x_2;\xi_2))  \, | \,    x_1=x_2, \ \xi_1=\xi_2 \}   \subset (T^*M \setminus \zero)^{\times 2}.
$$

\begin{definition}\label{deffey} 
	The  \emph{Feynman wavefront set} $\Lambda\subset (T^*M\setminus\zero)^{\times 2}$ is defined by
	$$
	\bea
	\Lambda & \defeq   \big( (\Sigma^+)^{\times 2} \cap   \{  ((x_1;\xi_1),\! (x_2;\xi_2))\, | \, (x_1;\xi_1)  \sim  (x_2;\xi_2) \mbox{ and } x_1\in J_-(x_2) \}\big)  \fantom \ \,  \cup  \big( (\Sigma^-)^{\times 2} \cap   \{  ((x_1;\xi_1),\! (x_2;\xi_2))\, | \, (x_1;\xi_1)  \sim  (x_2;\xi_2) \mbox{ and } x_1\in J_+(x_2) \} \big) \! \cup \Delta^*.
	\eea
	$$
\end{definition}

In the definition we employed the convention  which corresponds to considering  \emph{primed} wavefront sets (as opposed to wavefront sets of Schwartz kernels). We caution the reader that beside the choice of  working with `primed' or 'non-primed' wavefront sets, in the context of QFT there are two sign conventions possible.

\begin{definition}For $s\in\rr$, we write
	$$
	G_z = \Oreg{}
	$$ 
	if for all $l\in\rr$, $h(z)^{-1} G_z$ is a uniformly bounded family of continuous operators  $H^l_\c(M)\to H^{l+s}_\loc(M)$. We write $G_z=\Onorm{}$ if $G_z=\Oreg{}$ for some $s\in \rr$. 
\end{definition}

Note that $G_z=\Oreg{}$ implies  $G_z^*=\Oreg{}$, and  $G_z=\Onorm{}$ implies $G_z^*=\Onorm{}$.
\begin{definition}
	If $G_z=\Onorm{}$ then its \emph{uniform operator wavefront set of order $s\in\rr$} is the set $\wfl{}(G_z)\subset\p \overline{T^*M}\times \p \overline{T^*M}$ defined as follows:  $(q_1,q_2)\notin\wfl{}(G_z)$ iff there exist  $B_i\in \Psi^{0}(M)$, elliptic at $q_i$ ($i=1,2$) and such that 
	\beq\label{unireg}
	B_1 G_z  B_2^* = \Oreg{}.
	\eeq
\end{definition}

\begin{definition}\label{anewdef}
	Let $\kappa :T^*M\setminus\zero \to  \p \overline{T^*M}$ be the quotient map for the $\rr_{>0}$ action by fiberwise dilations.  For each conic set $\Lambda\subset T^*(M\times M)\setminus\zero$ we define
	$$
	\Lambda'=\{  (\kappa(x_1;\xi_1), \kappa(x_1;-\xi_2))    \st (x_1,x_2;\xi_1,\xi_2)\in \Lambda, \, \xi_1\neq 0, \ \xi_2\neq 0  \},
	$$ 
	which is a subset of $\p \overline{T^*M}\times \p \overline{T^*M}$.
\end{definition}

\begin{lemma}\label{lem:composition} Let $G_{1,z}=\Onormsh{}(h_1(z))$ and $G_{2,z}=\Onormsh{(h_2(z))}$ and suppose that the operators $G_{2,z}$ are proper{ly supported} for all $z\in Z$. Then the composition $G_{1,z}  G_{2,z}= \Onormsh{}(h_1 h_2(z))$ is well-defined and satisfies
	\beq\label{lklklklklk}
	\wfl{6}(G_{1,z}G_{2,z}) \subset \wfl{4}(G_{1,z}) \circ \wfl{5}(G_{2,z}),
	\eeq
	where the composition of $\Gamma_1,\Gamma_2\subset  \p \overline{T^*M}\times\p \overline{T^*M}$ is defined by
	$$
	\Gamma_1\circ \Gamma_2 = \{ (q_1,q_2)\in \p \overline{T^*M}\times\p \overline{T^*M} \st \exists q\in\p \overline{T^*M} \mbox{\,s.t.\,}  (q_1,q)\in \Gamma_1, \,   (q,q_2)\in \Gamma_2  \}.
	$$
\end{lemma} 
\begin{proof}   For all $A_1,A_2\in\Psi^0(M)$,
	\beq\label{eq:temokn}
	A_1 G_{1,z}G_{2,z} A_2^*= \textstyle\sum_{k} (A_1   G_{1,z} B_{k}^*) (B_k G_{2,z} A_{2}^*), 
	\eeq
	where   $B_k\in\Psi^0(M)$ is an arbitrary family such that $\sum_k  B_k^* B_k =\one$ as a locally finite sum. By taking $\wf'(B_k)$  sufficiently small and using \eqref{eq:temokn} we obtain \eqref{lklklklklk}.
\end{proof} 
The next lemma relates the notion of operatorial wavefront set with the wavefront set of the Schwartz kernel.

\bel\label{ditoop} Suppose $\Lambda\subset T^*(M\times M)\setminus\zero$ is conic and $G_z=\Onorm{}$. If the associated family of Schwartz kernels satisfies $G_z(\cdot)=\Olambda{}$ then $\wfl{}(G_z)\subset \Lambda' $
for all $s\in\rr$.
\eel 
\beproof {For ease of notation we  identify $T^*M\setminus\zero$ with   $S^*M$  using the quotient map $\kappa$.  Let $q=(x_1,x_2; \xi_1,\xi_2)\in T^*(M\times M)\setminus \Lambda$ with $\xi_1\neq 0$ and $\xi_2\neq 0$, and  let $\Gamma_i$, $i=1,2$,  be a small conic neighborhood of $(x_i;\xi_i)$, to be fixed later on. 
	Let $B_i\in\Psi^0(M)$ be elliptic at $(x_i;\xi_i)$, with $\wf'(B_i)\subset \Gamma_i$.  Let  $A\in\Psi^0(M\times M)$ be elliptic  on $\Gamma_1\times \Gamma_2$ and with symbol vanishing in a conical neighborhood of $\zero\times T^*M$ and $T^*M\times \zero$. This implies that $A(B_1\otimes B_2)\in \Psi^0(M)$  and that $A(B_1\otimes B_2)$ is elliptic at $q$. Since  $G_z(\cdot)=\Olambda{}$ and $q\notin \Lambda$,  we can take $\Gamma_1,\Gamma_2$ such that $\wf'(A(B_1\otimes B_2))$ is in a small enough neighborhood of $q$ so that $A(B_1\otimes B_2) G_z(x,y)=\Osmooth{}$. By ellipticity of $A$ on $\Gamma_1\times \Gamma_2$,  this  implies $(B_1\otimes B_2) G_z(x,y)=\Osmooth{}$, where $B_1,B_2$ acts on the first, resp.~second variable of the Schwartz kernel of $G_z$. Hence $B_1 G_z \bar{B}_2^*=\Oreg{}$ for all $s\in\rr$, where $\bar {B}_2^*\in \Psi^0(M)$ is defined via complex conjugation of the Schwartz kernel of  ${B}_2^*$ and is in consequence  elliptic at $(x_2;-\xi_2)$.  This implies $((x_1;\xi_1), (x_2;-\xi_2)) \notin \wfl{}(G_z)$ for all $s\in\rr$ as claimed. } \qed

\subsection{Uniform wavefront set estimate for the Feynman operator}

Until now we have worked in the conformal setting. To get back to the original space $M$ we define the operators
\beq
R^{(s_{\rm F/\bar F})}(\lambda):= x_0^{\frac{n-1}{4}-\frac{i\lambda}{2}}r_MP^{(s_{\rm F/\bar F})}(\lambda)^{-1}e_Mx_0^{\frac{i\lambda}{2}-\frac{n+3}{4}} :C^\infty_{\rm c}(M)\to \mathcal D'(M).
\eeq

We recall here the definition of the operatorial uniform wavefront set from \cite{DanWro}.
 \begin{definition} Let $Z\subset \cc$ and suppose $\{ G(z)\}_{z\in Z}$  is for all $m\in\rr$ a bounded family of operators in $B(H^m_\c(M),H^m_\loc(M))$.
	The \emph{uniform operator wavefront set of order $s\in\rr$ and weight $h(z)$} of $\{ G(z)\}_{z\in Z}$ is the set 
	\beq
	\WF'^{(s)}_{h(z)}\big( G(z) \big)\subset (T^*M\setminus\zero)\times (T^*M\setminus\zero)
	\eeq
	defined as follows:  $((x_1;\xi_1),(x_2;\xi_2))$ is \emph{not} in \eqref{eq:wfs} if and only if for all $\epsilon>0$ there exists a uniformly bounded  family $B_i(z)\in \Psi^{0}(M)$ of  properly supported operators, each elliptic at $(x_i;\xi_i)$ and such that for all $r\in \rr$, the family
	$$
	h(z)^{-1}B_1(z) G(z) B_2(z)^* \mbox{ for } z\in Z \mbox{ is bounded in } B(H^{r}_\c(M), H_\loc^{r+s}(M)).
	$$
	We write
	\beq
	\WF'^{}_{h(z)}\big( G(z) \big)=\cap_{s\in\rr}\WF'^{(s)}_{h(z)}\big( G(z) \big)
	\eeq
\end{definition}

\begin{proposition}\label{prop:wf} Let $C>0$ and $$Z_C= \{ \lambda\in\cc\st \vert\Im \lambda\vert\leq C, \ \lambda \ {\rm not \ a \ pole \  of \ } P^{s_{\rm F/\bar F}}(\lambda)^{-1}\}.$$ For some $s_{\rm F/\bar F}$ satisfying \eqref{varord}, the family  $\{ R^{(s_{\rm F/\bar F})}(\lambda)^{-1}\}_{\lambda \in  Z_C}$ has anti-Feynman wavefront set in the sense that it satisfies:
   \beq\label{feywf}
      \wf_{\bra \lambda\ket^{-1}}'\big( R^{(s_{\rm F})}(\lambda)^{-1} \big)\subset \{ (q_1,q_2)\in S^* M  \times S^* M   \st q_1 \succ q_2 \mbox{ or } q_1=q_2  \},  
   \eeq
   where by $q_1 \succ q_2$ we mean that $q_2$ can be reached from $q_1$ by a backward null bicharacteristic of $p$, and 
   \beq\label{antfeywf}
      \wf'_{\bra \lambda\ket^{-1}}\big( R^{(s_{\rm \bar F})}(\lambda)^{-1} \big)\subset \{ (q_1,q_2)\in S^* M  \times S^* M   \st q_2 \succ q_1 \mbox{ or } q_1=q_2  \},  
   \eeq
   \end{proposition}

 \begin{remark}
 	
\ben
\item We could prove a similar result for the $s_{\rm ftr}$ and $s_{\rm past}$ operators but we will see that because they extend the advanced and retarded propagators this property can be seen on the level of support properties rather than wavefront sets.
\item Considering the fact that the spectral parameter is $z=-\lambda^2-\frac{(n-1)^2}{4}$ the weights are actually  $\bra z\ket^{1/2}$ like in asymptotically Minkowski. However here our operator only makes sense in term of $\lambda$ and not $z$.
\een
 \end{remark}
  
   \begin{proof} We will prove the property for $s_{\rm F}$ and omit this subscript for simplicity of notations as well as omitting the $s_{\rm F}$ in the notation of the inverses. Let $(q_1,q_2)$ with $q_1\neq q_2$ and suppose $q_2$ cannot be reached from $q_1$ by a backward null bicharacteristic. We can find $B_1,B_2\in \Psi_{\rm c}^0(M)$ (meaning a pseudo-differential operator of order $0$ with compact support in $M$), formally self-adjoint and elliptic at resp.~$q_1,q_2$, such that every backward null bicharacteristic from $\wf'(B_1)$ is disjoint from $\wf'(B_2)$. We want to prove that for all $s_0,s_1\in\rr$,
   \beq\label{eq:bounded}
   B_1 R^{(s_{\rm F})} B_2 \in B(H^{s_0}_\c(M), H_\loc^{s_1}(M)).
   \eeq
   Let $A_0\in \Psi^{0}(X)$ be elliptic  at the source $L_-$ and  on  the backward null bicharacteristics  from $\wf'(B_1)$.   By the higher decay radial estimate, i.e.part (1) of ~Theorem \ref{semradest} , for $s\vert_{L_-}> S> \frac{1}{2}-\Im\lambda$ non-decreasing along the Hamilton flow and all $N$,  
   \beq\label{eq:} 
             \norm{B_1 u}_{H(X)^s}    \lesssim   \bra \Re\lambda\ket^{-1}  \| A_0P(\lambda) u \|_{H(X)^{s-1}}+ \bra\Re\lambda\ket^{-N}\norm{u}_{H(X)^{-N}}.
            \eeq
  In particular, we can take $s$ non-decreasing along the Hamiltonian flow such that $s\geq \max(s_1,\frac{1}{2}-\Im \lambda)$  on $\wf'(B_1)$ and $L_-$, $s-1\leq s_0$ on $\wf'(B_2)$, and $\wf'(A_0)\cap \wf'(B_2)=\emptyset$. Such an $s$ exists because backward bicharacteristics from $\wf'(B_1)$ do not reach $\wf'(B_2)$ and $s$ is non-decreasing from $\wf'(B_2)$ to $\wf'(A_0)$.  Applying this to $u=P(\lambda)^{-1} B_2 f$ for $f\in H_\c^{s_0}(M)$  yields
     \beq
               \norm{B_1 P(\lambda)^{-1} B_2 f}_{H(X)^s}    \lesssim    \bra\Re\lambda\ket^{-1} \| A_0 B_2 f \|_{H(X)^{s-1}}+\bra\Re\lambda\ket^{-N} \norm{P(\lambda)^{-1} B_2 f}_{H(X)^{-N}}.
               \eeq
 Since $A_0 B_2\in \Psi^{-\infty}_{}(X)$ and the last term can be estimated using invertibility  of $P(\lambda)$ conclude 
        \beq\label{wavefrontP}
                  \norm{B_1 P(\lambda)^{-1} B_2 f}_{H(X)^s}    \lesssim   \bra\Re\lambda\ket^{-1}  \|  f \|_{H(X)^{s_0}},
                  \eeq
 hence $  \norm{B_1 P(\lambda)^{-1} B_2 f}_{s_1}    \lesssim  \bra\Re\lambda\ket^{-1} \|  f \|_{s_0}$.
 Let us recall that formally
\beq P(\lambda)=x_0^{\frac{i\lambda}{2}-\frac{n+3}{4}}R(\lambda)x_0^{\frac{n-1}{4}-\frac{i\lambda}{2}}\eeq
on $M$. So for $\tilde B_1=B_1x_0^{\frac{i\lambda}{2}-\frac{n-1}{4}}$ and $\tilde B_2=x_0^{\frac{n+3}{4}-\frac{i\lambda}{2}}B_2$ we have 
\beq \norm{\tilde B_1R(\lambda)^{-1}\tilde B_2f}_{H^{s_1}(X)}\lesssim \bra\Re\lambda\ket ^{-1}\norm{f}_{H^{s_0}(X)}.\eeq
$\tilde B_1$ and $\tilde B_2$ are still bounded operators for $\Im \lambda$ fixed since they have compact support in $M$. And since $\tilde B_1, \tilde B_2$ and $f$ have compact support we obtain
\beq \norm{\tilde B_1R(\lambda)^{-1}\tilde B_2f}_{H^{s_1}(M)}\lesssim \bra\Re\lambda\ket ^{-1}\norm{f}_{H^{s_0}(M)}.\eeq
 \end{proof}
  
  \begin{center}
  	\includegraphics[scale=0.5]{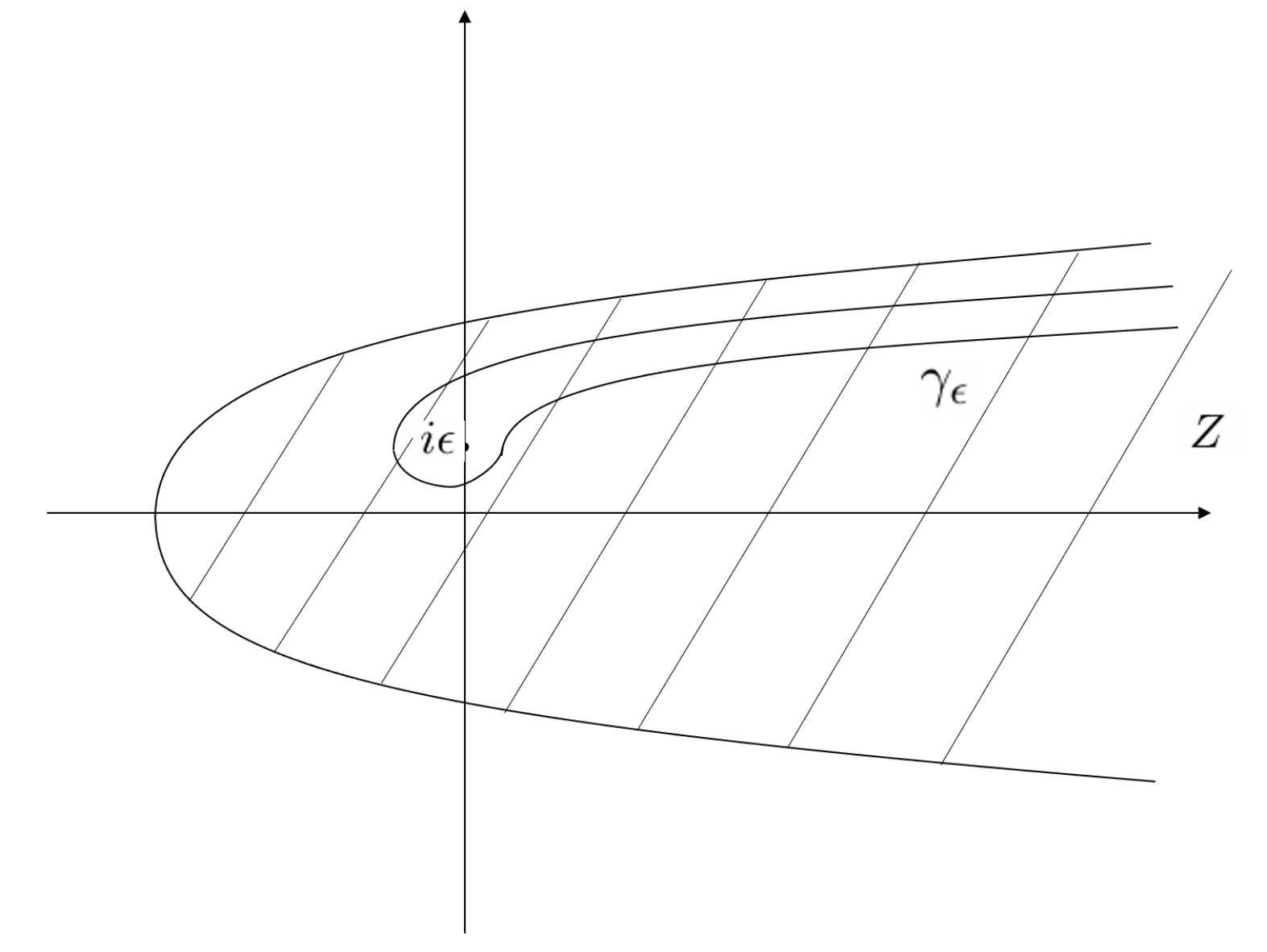}
  	
  	\bfseries Figure 3. \normalfont  Representation of the set $Z$ and the curve $\gamma_\epsilon$ in the coordinates $\lambda^2+\frac{(n-1)^2}{4}=-z$. Here $Z=Z_C\setminus Z_c$ for $C>c$.
  \end{center}
  
 \begin{lemma}\label{lem:pika} Assume that $(M,g)$ is an even asymptotically de Sitter non-trapping at all energies manifold and $\epsilon>0$.  Then  $\{ R^{s_{\rm F}}(\lambda)\}_{\lambda\in Z\setminus \{poles\}}$ can be restricted to an operator $R_F(\lambda):H_{\rm c}^{s}\to H_{\rm loc}^{\frac{1}{2}\Im\lambda-\epsilon-1}(M)$ for all $s\in \rr$, $\epsilon>0$ that satisfies 
 	\beq\label{resolventbehaves}
 	R^{s_{\rm F}}(\lambda)=\pazocal O_{H_{\rm c}^{s+1}\to H_{\rm loc}^{\inf{(s,\frac{1}{2}-c)}-\epsilon-1}(M)}{(\bra \lambda\ket^{-1})}
 	\eeq
 	since $\Im\lambda>c$ on $Z$.
 	And for $\psi \in C^\infty_{\rm c}$ we have for all $s\in \rr$ and $\epsilon>0$
 	\beq
 	\psi R^{s_{\rm F}}(\lambda)=\pazocal O_{H_{\rm c}^{s}\to H_{\rm loc}^{s-\epsilon-2}(M)}{(\bra \lambda\ket^{-1})}
 	\eeq
 	so 	$\psi R^{s_{\rm F}}(\lambda)=\pazocal O_{C^\infty \to C^\infty}{(\bra \lambda\ket^{-1})}$.
 \end{lemma}
 
 \begin{proof}
 By \eqref{invest1} $R^{s_{\rm F}}(\lambda)$ is $\pazocal O(\bra \lambda\ket^{-1})$ for $z\in \cc$ outside of the poles (which are in finite number in $Z$) as a bounded map $\pazocal X_{\rm c}^{s_{\rm F/\bar F}}\to H_{\rm loc}^{s_{\rm F/\bar F}-1}(M)$, where $\pazocal X_{\rm c}^{s_{\rm F/\bar F}}$ are the element of $\pazocal X^{s_{\rm F/\bar F}}$ with compact support in $M$. Since $H_{\rm c}^{s_{\rm F/\bar F}-1}(M)\subset\pazocal X_{\rm c}^{s_{\rm F/\bar F}}$ we have that it is bounded from $H_{\rm c}^{\sup s_{\rm F/\bar F}+1}(M)$ to $H_{\rm loc}^{\inf s_{\rm F/\bar F}-1}(M)$ but since the element of $H_{\rm c}^{\sup s_{\rm F/\bar F}+1}(M)$ have compact support we can actually take $\sup s_{\rm F/\bar F}$ only on the support of the distribution in question where it be anything. Then $\inf s_{\rm F/\bar F}$ just as to be lower that $\sup s_{\rm F/\bar F}$ and the threshold. However with the cutoff we do not have to worry about the regularity of the image being lower that the threshold.
 \end{proof}

\section{Hadamard parametrix}\label{Had}

\subsection{Definition of the Hadamard parametrix}

In this section we will follow closely the definition of the Hadamard parametrix given in \cite{Dang2021} but the proofs can be found in \cite{DanWro}.

Let $P_0=\square_g$ be the Laplace--Beltrami operator on a time-oriented manifold $(M,g)$ of \emph{even} dimension $n$. We will consider the operator $P_0-z$ for $\Im z\geq 0$

The Hadamard parametrix for $P_0-z$ can be constructed four steps.\\

\emph{Step 1.} Let $\eta=dx_0^2-(dx_1^2+...+dx_{n-1}^2)$ be the usual metric on  Minkowski space $(\rr^n,\eta)$. We consider the associated quadratic form
\beq
\vert \xi\vert_\eta^2=-\xi_0^2+\sum_{i=1}^{n-1}\xi_i^2
\eeq
that we define with a minus sign for convenience. The distribution $(\vert \xi\vert_\eta^2-z)^{-\alpha}$ is well defined for $\alpha\in \cc$ and $\Im z>0$. In general, if $\Im z\geq 0$, the limit from the upper half-plane $(\vert \xi\vert_\eta^2-z-i0)^{-\alpha}=\lim_{\epsilon\to 0^+}(\vert \xi\vert_\eta^2-z-i\epsilon)^{-\alpha}$ is a well defined distribution on $\rr^n\setminus\{0\}$ and it is a holomorphic family of distribution is $\alpha\in\cc$ for $z\neq 0$ (meromorphic family for $z=0$). We define its Fourier transform (with appropriate normalization)
\beq\label{dishad}
F_\alpha(z,x):=\frac{\Gamma(\alpha+1)}{(2\pi)^n}\int e^{i\bra x,\xi\ket}(\vert \xi\vert_\eta^2-i0-z)^{-\alpha-1}d^n\xi.
\eeq
It is a family of distributions on $\rr^n$ holomorphic in $\alpha\in \cc\setminus \{-1,-2,...\}$ if $\Im z\geq 0$, $z\neq0$.\\

\emph{Step 2.}
We then use the exponential map to pull back $F_\alpha(z,x)$ on a neighborhood of the diagonal $\Delta\subset M\times M$.

More precisely, Let $\exp_x:T_xM\to M$ be the exponential map of geodesics. Take a neighborhood of the zero section $o$ in $TM$ small enough such that the map 
\beq\label{exp}
(x;v)\mapsto (x,\exp_x(v))\in M^2
\eeq 
is a local diffeomorphism onto its image that we denote $\pazocal U$. Let $(e_1,...,e_n)$ be local time-oriented orthogonal frame that we define on an open set and $(s^i)_{i=1}^n$ its associated coframe. Let $(x_1,x_2)\in \pazocal U$ such that $x_1,x_2$ are in that open set, we define
\beq
G:(x_1,x_2)\mapsto \big( G^i(x_1,x_1)=s_{x_1}^i(\exp_{x_1}^{-1}(x_2))_{i=1}^n\in \rr^n. 
\eeq
The map $(x_1,x_2)\mapsto (x_1,\exp_{x_1}(x_2))$ is a diffeomorphism since it is the inverse of \eqref{exp}, therefore $G$ is a submersion.

Let $f\in \pazocal D'(\rr^n)$, we consider the pull-back $(x_1,x_2)\mapsto G^*f(x_1,x_2)$. If $f$ is invariant under the connected component of the identity of the Lorentz group $O(1,n-1)$ the pull-back does not depend on the choice of $(e_\mu)_\mu$. Therefore we can define $G^*f\in \pazocal D'(\pazocal U)$ canonically and $\pazocal U$ is a neighborhood of $\Delta$.

\begin{definition}
    For $\alpha\in \cc$, the distribution $\textbf{F}_\alpha(z,.)=G^*{F}_\alpha(z,.)\in \pazocal D'(\pazocal U)$ is defined by pull-back of the $O(1,n-1)$ invariant distribution $F_\alpha(z,.)\in \rr^n$ introduced in \eqref{dishad}
\end{definition}
\emph{Step 3.}
We then construct the Hadamard parametrix using normal charts and the $\textbf{F}_\alpha(z,.)$. More precisely, we fix $x_0\in M$ and define the distribution
\beq
\textbf{F}_\alpha(z,\vert.\vert_g):=\textbf{F}_\alpha(z,x_0,\exp_{x_0}(.))\in \pazocal D'(U)
\eeq
for some $U\subset T_{x_0}M$. We then look for a parametrix of order $N$, $H_N(z)$ with the form
\beq
H_N(z)=\sum_{k=0}^Nu_k\textbf{F}_\alpha(z,\vert.\vert_g),
\eeq
solving 
\beq\label{par1}
(P-z)H_N(z)=\vert g\vert^{-\frac{1}{2}}\delta_0+(Pu_N)\textbf{F}_N.
\eeq
 We find, after computations, that the $(u_k)_{k=0}^\infty$ taken as function on $C^\infty(U)$ should solve the hierarchy of transport equations
\beq
2ku_k+b^i(x)\eta_{ij}x^ju_k+2x_i\p_{x_i}u_k+2Pu_{k-1}=0
\eeq
where $b^i(x)=\vert g(x)\vert^{-\frac{1}{2}}g^{jk}(x)(\p^j\vert g(x)\vert^{\frac{1}{2}})$, we sum over repeated indices and we take as initial conditions $u_{-1}=0$, $u_0=1$. It happens that the transport equations have a unique solution.

\emph{Step 4.} Finally we will consider the dependence on $x_0$ of the parametrix in order to get a parametrix on $\pazocal U$. Since we can show $\pazocal U \ni (x_1,x_2)\mapsto u_k(s(\exp_{x_1}^{-1}(x_2)))$ is smooth for all $k$, and $\textbf{F}(z.)$ is well defined on $\pazocal U$,
\beq
H_N(z,x_1,x_2)=\sum_{k=0}^Nu_k(s(\exp_{x_1}^{-1}(x_2)))\textbf{F}_\alpha(z,x_1,x_2)
\eeq
is a well defined distribution on $\pazocal U$. 
\begin{definition}
Without writing the exponential in order to be more concise, the \emph{Hadamard parametrix} of order $N$ is 
\beq
H_N(z)=\sum_{k=0}^Nu_k\textbf{F}_\alpha(z,.)\in \pazocal D'(\pazocal U).
\eeq
\end{definition}

Using a cutoff function $\chi\in C^\infty(M^2)$ supported in $\pazocal U$, we can extend $H_N(z,.)$ to $M^2$
\beq
H_N(z)=\sum_{k=0}^N\chi u_k\textbf{F}_\alpha(z,.)\in \pazocal D'(M\times M).
\eeq

Let $\tilde{\pazocal{U}}$  be  
a neighborhood of the 
diagonal $\diag$ such that $\chi|_{\tilde{\pazocal{U}}}=1$ and $\tilde{\pazocal{U}}\subset \pazocal{U}$.
Let 
$$
\Lambda_{\chi}=\{(x,y;\xi,\eta) \st (x,y;\xi,\eta)\in \Lambda, \, (x,y)\notin \tilde{\pazocal{U}}  \},
$$
so  $\Lambda_{\chi}=\Lambda\cap T^*\big((M\times M)\setminus\tilde{\pazocal{U}} \big)$ is a 
\emph{truncation of the Feynman wavefront set} $\Lambda$ where we removed the 
neighborhood $\tilde{\pazocal{U}}$
near the diagonal. 

All these constructions, together with estimates from \cite{DanWro}, lead to the following Lemma:
\begin{lemma}\label{lem:had}
Assume  $(M,g)$ is globally hyperbolic and set $P=\square_g+m^2$.
For every $s\in \mathbb{R}_{\geqslant 0}$, $p\in \mathbb{Z}_{\geqslant 0}$, $m\geqslant 0$ there exists $N$ large enough s.t.
for every $z\in\{\Im z\geqslant 0,\, \vert z\vert\geqslant \epsilon>0\}$
\begin{eqnarray}
	\left(P-z\right)\left(\sum_{k=0}^N u_k \mathbf{F}_k(z-m^2,.)\chi\right)={ \module{g}^{-\frac{1}{2}}}\delta_{\Delta}+(Pu_N)\mathbf{F}_N(z-m^2,.)\chi+r_N(z),
\end{eqnarray}
where:
\ben
\item\label{manul1} ${\module{g}^{-\frac{1}{2}}}\delta_{\Delta}\in \pazocal{D}^\prime(M\times M)$ is the Schwartz kernel of the identity map,
\item\label{manul2} $\bra z\ket^{p}(Pu_N)\mathbf{F}_N(z-m^2,.)\chi$ is bounded in $\pazocal{C}^s_{\loc}(\pazocal{U})$, 
\item\label{manul3}
$r_N(z,.)\in \pazocal{D}^\prime(M\times M)$ vanishes on $\tilde{\pazocal{U}}\subset\pazocal{U}$ and outside $\pazocal{U}\subset M\times M$. In particular, $r_N(z,.)$ is the Schwartz kernel of a family of {proper{ly supported } operators}. Furthermore,  $r_N(z,.)$  is bounded in $\pazocal{D}^\prime_{\Lambda_{\chi}}(M\times M)$
uniformly in $z\in \{\Im z\geqslant 0, \,\vert z\vert\geqslant \epsilon>0\}$, and $ r_N(z,.)=\pazocal{O}_{\pazocal{D}^\prime_{\Lambda_{\chi}}}( \braket{\Im z}^{-\infty} )  .$
\een
\end{lemma}

\subsection{The Hadamard parametrix and the spectral action}

We will first start by justifying our definition of complex powers.

\begin{lemma}\label{selcompow}
Let $T$ be a selfadjoint operator on $\pazocal H$, $\gamma_\epsilon$ be the contour defined by
\beq
\bea
\gamma_\epsilon: \rr &\to \cc,\\
t\in (-\infty,-\eta] &\mapsto -(t+e^{i\pi/4}\sqrt{\epsilon}+iC)^2-\frac{(n-1)^2}{4},\\
t\in (-\eta,\eta) &\mapsto h_\epsilon(t),\\
t\in [\eta,+\infty] &\mapsto -(t+e^{i\pi/4}\sqrt{\epsilon}ic)^2-\frac{(n-1)^2}{4},\\
\eea
\eeq
where $\eta>0$, $C>c>0$ and $(h_\epsilon(t))_{\epsilon\geq 0}$ is a smooth family of curves from $\gamma_\epsilon (-\eta)$ to $\gamma_\epsilon (\eta)$ surrounding $i\epsilon$ without going through the real axis (see Figure 3).

Then we have that $(T-i\epsilon)^{-\alpha}$, which is defined by functional calculus, is equal to
\beq
(T-i\epsilon)^{-\alpha}=\frac{1}{2\pi i}\int_{\gamma_\epsilon}(z-i\epsilon)^{-\alpha}(T-z)^{-1}dz.
\eeq
for $\Re\alpha\geq 1$. Or, for $\Re \alpha\geq 1-N$
\beq
(T-i\epsilon)^{-\alpha}=(T-i\epsilon)^{N}\frac{1}{2\pi i}\int_{\gamma_\epsilon}(z-i\epsilon)^{-\alpha-N}(T-z)^{-1}dz.
\eeq
\end{lemma}

\begin{proof}
We will focus on the case $N=0$ since it works similarly for $N\neq 0$.

\emph{Step 1.}	We will first show that the integral is convergent in the strong operator topology.

\beq
\norm{\int_{\gamma_\epsilon}(z-i\epsilon)^{-\alpha}(T-z)^{-1}dz}\leq \int_{\gamma_\epsilon}\vert (z-i\epsilon)^{-\alpha}\vert \norm{(T-z)^{-1}}dz
\eeq
Since $\vert\vert (T-z)^{-1}\vert\vert\leq \vert{\Im z}\vert^{-1}$, $\alpha\geq 1$,
\beq
\vert (z-i\epsilon)^{-\alpha}\vert \norm{(T-z)^{-1}}\leq \frac{\Re z}{\Im z}.
\eeq
Therefore
\beq
 \int_{\gamma_\epsilon}\vert (z-i\epsilon)^{-\alpha}\vert \norm{(T-z)^{-1}}dz\leq \int_\rr \vert \frac{1}{\Re \gamma_\epsilon(t)\Im \gamma_\epsilon(t)}\vert \ \vert \gamma'_\epsilon(t)\vert dt
\eeq
For $t>>1$ we have
\beq
\vert\frac{1}{\Re \gamma_\epsilon(t)\Im \gamma_\epsilon(t)}\vert \ \vert \gamma'_\epsilon(t)\vert
\eeq
and $-(t+e^{i\pi/4}\sqrt{\epsilon}+ic)^2-\frac{(n-1)^2}{4}$ so $\vert \Re\gamma_\epsilon(t)\vert \sim t^2$, $\vert\Im \gamma_\epsilon(t)\vert \sim (c+\frac{\epsilon}{\sqrt 2})t\vert$ and $\vert \gamma'_\epsilon(t)\vert\sim t$ so 
\beq
\vert\frac{1}{\Re \gamma_\epsilon(t)\Im \gamma_\epsilon(t)}\vert \ \vert \gamma'_\epsilon(t)\vert\sim \frac{1}{(c+\frac{\epsilon}{\sqrt 2})t^2}
\eeq
which is integrable.

\emph{Step 2.} Let us assume $T$ is bounded. Let $\mathcal C_R$ be the circle of radius $R$ centered on the origin with large $R>0$ and let $t_1(R)$, $t_2(R)$ be the two points where $\gamma_\epsilon$ intersect $\mathcal C_R$. We define the curve $\gamma_{\epsilon,R}:\mathbb S^1\to\cc$ that is $\gamma_\epsilon\vert_{[t_1(R),t_2(R)]}$ from $\pi$ to $2\pi$ and the long arc circle joining $\gamma_\epsilon(t_2(R))$ to $\gamma_\epsilon(t_1(R))$ from $0$ to $\pi$. We have by functional calculus and complex analysis that
\beq
(T-i\epsilon)^{-\alpha}=\frac{1}{2\pi i}\int_{\gamma_{\epsilon,R}}(z-i\epsilon)^{-\alpha}(T-z)^{-1}dz.
\eeq
Let $\mathcal C'_{\epsilon,R}:= \gamma_{\epsilon,R}\vert_{[0,\pi]}$. We want to show that $\int_{\mathcal C'_{\epsilon,R}}(z-i\epsilon)^{-\alpha}(T-z)^{-1}dz\to 0$ as $R\to\infty$ since if it is we have that
$$
\lim_{R\to \infty}\int_{\gamma_{\epsilon,R}}(z-i\epsilon)^{-\alpha}(T-z)^{-1}dz=\int_{\gamma_{\epsilon}}(z-i\epsilon)^{-\alpha}(T-z)^{-1}dz.
$$
because we know it is integrable. Let $l_\epsilon(R)\leq 2\pi R$ be the length of the arc circle, then
$$
\int_{\mathcal C'_{\epsilon,R}}(z-i\epsilon)^{-\alpha}(T-z)^{-1}dz=\int_{0}^1(Re^{i (\theta+\theta_0(R)) l_\epsilon(R)}-i\epsilon)^{-\alpha}(T-Re^{i (\theta+\theta_0(R)) l_\epsilon(R)})^{-1}l_\epsilon(R)d\theta
$$
So
$$
\vert \int_{\mathcal C'_{\epsilon,R}}(z-i\epsilon)^{-\alpha}(T-z)^{-1}dz\vert \leq \int_{0}^1\vert(Re^{i (\theta+\theta_0(R)) l_\epsilon(R)}-i\epsilon)^{-\alpha}\vert\frac{1}{{\rm dist}(Re^{i (\theta+\theta_0(R))},{\rm Spec}{(T)})}2\pi Rd\theta
$$
but $\vert(Re^{i (\theta+\theta_0(R)) l_\epsilon(R)}-i\epsilon)^{-\alpha}\sim R^{-1}$ and ${\rm dist}(Re^{i (\theta+\theta_0(R))},{\rm Spec}{(T)})\sim R$ so
$$
\vert \int_{\mathcal C'_{\epsilon,R}}(z-i\epsilon)^{-\alpha}(T-z)^{-1}dz\vert \leq C\int_{0}^1 R^{-1}d\theta
	\to 0.
$$
	
\emph{Step 3.}
Let $T$ be unbounded, $1_I$ be the unitary step function on $I$ a finite interval and $A=1_I(T)T$. We have
\beq
(A-i\epsilon)=\frac{1}{2\pi i}\int_{\gamma_\epsilon}(z-i\epsilon)^{-\alpha}(A-z)^{-1}dz.
\eeq
and therefore since $[1_I(T),T]=0$,
\beq
1_I(T)(T-i\epsilon)=1_I(T)\frac{1}{2\pi i}\int_{\gamma_\epsilon}(z-i\epsilon)^{-\alpha}(T-z)^{-1}dz
\eeq
for all $I$. Therefore it must be true for $I=\rr$.
	\end{proof}
	
This formula can be extended by taking $A=T+\mu$, $\mu\in \rr$ we namely extend the formula to
\beq
(A-\mu-i\epsilon)=\frac{1}{2\pi i}\int_{\gamma_\epsilon}(z-i\epsilon)^{-\alpha}(A-\mu-z)^{-1}dz.
\eeq
	
Now that we have shown our definition of complex powers is justified let us define the complex powers of the Hadamard parametrix.
\beq
\bea
H_N^{(\alpha)}(\mu+i\epsilon,.):=&\frac{1}{2\pi i}\int_{\gamma_\epsilon}(z-i\epsilon)^{-\alpha}H_N^{}(z+\mu,.)dz\\
&=\sum_{k=0}^N\chi u_k\frac{1}{2\pi i}\int_{\gamma_\epsilon}(z-i\epsilon)^{-\alpha}\textbf F_k(z+\mu,.)dz.
\eea
\eeq
The integral converges for $\Re \alpha>0$ by  \cite[Lemma 6.1]{DanWro} and we can evaluate the integral using the identity
\beq
\frac{1}{2\pi i}\int_{\gamma_\epsilon}(z-i\epsilon)^{-\alpha}\textbf F_k(\mu+z,.)dz=\frac{(-1)^k\Gamma(-\alpha+1)}{\Gamma(-\alpha-k+1)\Gamma(\alpha+k)}\textbf F_{k+\alpha-1}(\mu +i\epsilon,.)
\eeq
proven in Section 7.1 of \cite{DanWro}. By analytic continuation of $H_N^{(\alpha)}(\mu+i\epsilon,.)$ we have that
\beq
H_N^{(\alpha)}(z,.)=\sum_{k=0}^N\chi u_k\frac{(-1)^k\Gamma(-\alpha+1)}{\Gamma(-\alpha-k+1)\Gamma(\alpha+k)}\textbf F_{k+\alpha-1}(z,.).
\eeq
The following theorem is issued from and proven in \cite{DanWro}.

\begin{theorem}\label{mainmainthm}
	Let $(M,g)$ be a Lorentzian manifold. 
	The restriction of the Hadamard parametrix on the diagonal
	$H^{(\alpha)}(\mp i \epsilon,x,x)$ exists and is holomorphic for $\Re\alpha>\frac{n}{2} $, and it extends as a meromorphic 
	function of $\alpha$ with simple poles along the arithmetic progression
	$\{\frac{n}{2}, \frac{n}{2}-1,\dots,1\}$. Furthermore,
	$$
	\bea 
	\lim_{\epsilon\rightarrow 0^+} {\rm res}_{\alpha=\frac{n}{2}-\varm}H^{(\alpha)}(\mp i \epsilon,x,x)= 
	\mp\frac{ i \,u_\varm(x,x)}{2^n\pi^{\frac{n}{2}}(\frac{n}{2}-\varm-1)!}.
	\eea $$
	In particular,
	$$
	\bea 
	\lim_{\epsilon\rightarrow 0^+} {\rm res}_{\alpha=\frac{n}{2}-1}H^{(\alpha)}(\mp i \epsilon,x,x)= 
	\pm\frac{ i \,R_g(x)}{6(4\pi)^{\frac{n}{2}}(\frac{n}{2}-2)!}
	\eea 
	$$
	where $R_g(x)$ is the scalar curvature at $x$.
\end{theorem}

\section{The Feynman propagator and the "spectral action"}\label{Fey}
\subsection{Approximation of the Feynman propagator by the Hadamard parametrix}

We first need to translate the parameter $\lambda$ to the spectral parameter $z$ by defining
\beq
\lambda(z)=\sqrt{-z-\frac{(n-1)^2}{4}},
\eeq
where the square root is the one send $\cc$ to the upper half plane. Since $z=-\frac{(n-1)^2}{4}-\lambda^2$. With this definition we have
  \beq
\WF_{\bra z\ket^{\frac{1}{2}}}'\big( R_F(\lambda)^{-1} \big)\subset \{ (q_1,q_2)\in S^* M  \times S^* M   \st q_1 \succ q_2 \mbox{ or } q_1=q_2  \}. 
\eeq

\begin{proposition}\label{l:decompositionresolvent}
	Assume that $(M,g)$ is a non-trapping even asymptotically de Sitter space and let $\epsilon>0$.  For every $s\in \mathbb{R}_{\geqslant 0}$, $p\in \mathbb{Z}_{\geqslant 0}$, $m\geqslant0$, 
	there exists $N$ large enough s.t.~uniformly in $z\in\gamma_\epsilon$, we have the identity 
	\begin{eqnarray}\label{eq:toinsert}
		R_F(\lambda(z))=\left(\sum_{k=0}^N u_k \mathbf{F}_k(z-m^2,.)\chi\right)+E_{N,1}(z)+E_{N,2}(z)
	\end{eqnarray}
	in the sense of operators,
	 where $E_{N,1}(z)=R_F(\lambda(z)) r_N(z)$ satisfies
	
	\beq\label{manixmanix}
	\wfl{7}(E_{N,1}(z))\subset \Lambda_{\chi}',
	\eeq
	and $E_{N,2}(z)=R^{(s_{\rm F})}(\lambda(z))(Pu_N)\mathbf{F}_N(z)\chi$ satisfies $
	\psi E_{N,2}(z)=\Oregsh({\braket{z}^{-p}})$ for all $\psi\in C^\infty_{\rm c}$. 
\end{proposition}

In particular, \eqref{manixmanix} implies that $E_{N,1}(z)$  is \emph{smooth near the diagonal}.

As  explained earlier, there are inevitable losses in decay in $z$ in the high regularity estimates, and the $\pazocal{O}(\bra z\ket^{-p})$ bound requires to choose $N$ extremely large. 
We adapt the proof of Lemma \ref{l:decompositionresolvent} from \cite{DanWro} to our situation.

\begin{refproof}{Proposition \ref{l:decompositionresolvent}}
	Recall that   $\psi R_F(\lambda(z))=\pazocal O_{C^\infty\to C^\infty}(\bra z \ket^{-\frac{1}{2}})$  along $\gamma_\epsilon$ by Lemma \ref{prop:wf}.  
	{  
		By \eqref{manul2} of Lemma \ref{lem:had}, for every $(p,a)\in \mathbb{N}\times \mathbb{R}$, for $N$ large enough, {in $\cU$} we have $\braket{z}^{p} (Pu_N)\mathbf{F}_N(z)\chi\in H_{\loc}^{a+n-0}(M\times M)$ uniformly in $z$ along $\gamma_\epsilon$ 
		by the Sobolev embeddings $\pazocal{C}_{\loc}^{a} (M\times M)\hookrightarrow H_{\loc}^{a+n-0}(M\times M)$ recalled in Lemma~\ref{lem:holdersobolev} in the appendix. 
		For every $b\in \mathbb{R}$, $s\in \mathbb{R}_{\geqslant 0}$, the exterior product $ H^b_{\c}(M)\times H_\c^{-b-s}(M)\ni (v_1,v_2)\mapsto v_1 \otimes v_2\in H^{\inf(b,-b-s,-s)}_\c(M\times M)$
		is linear continuous~\cite[Thm.~3.2 p.~140]{joshi}, therefore choosing $N$ large enough so that $a+n+\inf(b,-b-s,-s)>0 $, we find that
		$$ 
		\bea
		H^b_{\c}(M)\times H_\c^{-b-s}(M)\ni(v_1,v_2)\mapsto &\left\langle v_2, \braket{z}^{p}(Pu_N)\mathbf{F}_N(z)\chi v_1\right\rangle_{M}\\ &= \left\langle \braket{z}^{p}(Pu_N)\mathbf{F}_N(z)\chi, v_1\otimes v_2\right\rangle_{M\times M}
		\eea
		$$
		is bilinear continuous by Sobolev duality uniformly in $z$ along $\gamma_\epsilon$. 
		Therefore we have (for $N$ large enough):
		\beq\label{njn}
		(Pu_N)\mathbf{F}_N(z)\chi = \Oregsh({\braket{z}^{-p}}).
		\eeq
		Since the operators in \eqref{njn} are proper{ly supported},
		$$
		\psi E_{N,2}(z)=\psi R_F(\lambda(z))(Pu_N)\mathbf{F}_N(z)\chi=\Oregsh({\braket{z}^{-p}})$$ follows by composition.
	}

	Next, by \eqref{manul3} of Lemma \ref{lem:had}, $r_N(z, .)=\pazocal{O}_{\pazocal{D}^\prime_{\Lambda_{\chi}}}( \braket{ z}^{-\infty} )$ along $\gamma_\epsilon$. Consequently,  in terms of the operator wavefront set  we have $\wfl{7}(r_N(z))\subset \Lambda_{\chi}'$ by Lemma \ref{ditoop}. We also know from Theorem \ref{prop:wf} that
	\beq
	\wfl{12}  \big( \psi R_F(\lambda(z)) \big)\subset \Lambda'. 
	\eeq
	By the composition rule for operator wavefront sets, i.e., by lemma \ref{lem:composition}, we obtain
	\beq\label{manix1}
	\wfl{7}(\psi E_{N,1}(z))\subset \Lambda' \circ \Lambda_{\chi}'.
	\eeq
	It is easy to show using the {transitivity} of the $q_1{\succ} q_2$ relation that $\Lambda$ and $\Lambda_\chi$ satisfy the remarkable property
	\beq\label{manix2} 
	\Lambda' \circ \Lambda'\subset \Lambda' \mbox{ and }  \Lambda' \circ \Lambda_{\chi}'\subset \Lambda_{\chi}'.
	\eeq
	From \eqref{manix1} and the second property in \eqref{manix2} we conclude \eqref{manixmanix} immediately.  The case of $\gamma_0$ is fully analogous. Since it is true for all $\psi$,
	\beq\label{manix1}
	\wfl{7}(E_{N,1}(z))\subset \Lambda' \circ \Lambda_{\chi}'.
	\eeq
\end{refproof}

\subsection{Complex powers of the Feynman propagator}
We want to define the complex powers of the Feynman propagator like we did for the Hadamard parametrix. We highlight however that even though it is justified by the selfadjoint case it is not clear that the integral defines powers of the Feynman resolvent, the name is therefore mostly an analogy.
\begin{definition}
	For $\alpha \geq 1$, $N\in \mathbb N_0$, and $\gamma_\epsilon$ avoiding the poles of $R^{(s_{\rm F})}(\lambda(z+\mu))$ (which is possible since we do not have poles for large $\Re\lambda)$, we define the complex powers of the Feynman propagator to be
	\beq
	R_F^{(\alpha,k)}(\mu+i\epsilon,.)=\frac{1}{2\pi i}P_0(\lambda(\mu+i\epsilon))^k\int_{\gamma_\epsilon}(z-i\epsilon)^{-\alpha-k}	R_F(\lambda(z+\mu))dz.
	\eeq
	\end{definition}
	
We need to check that the integral is convergent. For this we use equation $\eqref{invest1}$ to see that $\norm{R_F(\lambda(z+\mu))}\leq \bra \Re\lambda(z+\mu)^{-1}\ket$ which is a similar decay to the resolvent in the proof of Lemma \ref{selcompow} so it is integrable for the same reasons.
\begin{remark}
Note that the complex powers are not defined uniquely for each $\alpha$ in this case. We could fix $k=0$ but not doing so will give us a family of spectral actions.
\end{remark}
We then obtain decomposition of the Schwartz kernel of $R^{(s_{\rm F},\alpha,k)}$ in term of germ of distributions near the diagonal:

\begin{lemma} \label{l:decfeynmanpowers}  Let $(M,g)$ and $\psi$ be as in Proposition \ref{l:decompositionresolvent}. Then for every $l\in \mathbb{R}_{\geqslant 0}$, $p\in \mathbb{N}$, there exists $N\geqslant 0$ s.t.~we have the decomposition in $\pazocal{D}^\prime(\tilde{\pazocal{U}})$:
	\begin{eqnarray}
		R_F^{(\alpha,k)}(\mu+i\epsilon,.)=H_N^{(\alpha)}(\mu+i\epsilon,.)+R_N(i\epsilon,\alpha,k) \\
		{ R_N(i\epsilon,\alpha,k)\in C^l(\pazocal{U})}
	\end{eqnarray}
	where $\psi$ multiplied to the terms on the r.h.s.~depend 
	holomorphically on $\alpha$ in the half-plane $\Re\alpha>2-k-p$ and $\psi$ multiplied to the terms on the r.h.s.~is well-defined as an element of $\pazocal{D}^\prime(\tilde{\pazocal{U}})$.
\end{lemma} 

\begin{proof}
	The decomposition is direct from proposition \ref{l:decompositionresolvent} and the regularity of $H_N^{(\alpha)}(z,.)$ is clear from the definition. We are left to show that $R_N(i\epsilon,\alpha)$ is holomorphic and in $C^k(\pazocal U)$.
	
	Choose $N$ so that $s$ from \ref{l:decompositionresolvent} is strictly larger than $l+\frac{n}{2}$. We have from Lemma \ref{sobsch} that the Schwartz kernel of $E_{N,2}$ is in $H^s(M\times M)$ with norm $\pazocal O (\bra z\ket ^{-p})$. 
	
	Let $\phi$ be a smooth cutoff function. $\phi E_{N,1}\phi\in C^\infty(M\times M)$ and $\norm{\phi E_{N,1}\phi}_{H^s}= \pazocal O (\bra z\ket ^{-\infty})$ since 	$\wfl{7}(E_{N,1}(z))\subset \Lambda_{\chi}'$. So $\phi E_{N,1}+E_{N,2}\phi$ is $C^s(\mathcal U)$ with $H^s$ norm $\pazocal O (\bra z\ket ^{-p})$.
	\beq
	\phi R_N(i\epsilon,\alpha,k)\phi=\frac{1}{2\pi i}P_0(\lambda(\mu+i\epsilon))^k\int_{\gamma_\epsilon}(z-i\epsilon)^{-\alpha-k} \phi E_{N,1}(z)+E_{N,2}(z)\phi	dz.
	\eeq
	
And since 
\beq
\norm{(z-i\epsilon)^{-\alpha-k} \phi E_{N,1}(z)+E_{N,2}(z)\phi}_{H^s}=\pazocal O(z^{-\alpha-k-p})
\eeq
we have that it (and all its derivatives in $\alpha$) is integrable in $H^s$ for $-\Re \alpha-k-p<-2$. By Sobolev embedding that can be found in \cite{Heb} we have that $H^s$ is continuously included in $C^l$ and therefore $\phi R_N(i\epsilon,\alpha,k)\phi$ is holomorphic in $C^l(\pazocal U)$. Since it is true for all $\phi$ we have that $ R_N(i\epsilon,\alpha,k)$ is holomorphic in $C^l(\pazocal U)$
\end{proof}

\subsection{Feynman spectral action}

We now have all the pieces to formulate the spectral action theorem: 
\begin{theorem}\label{mainmainthm}
	Let $(M,g)$ be a non-trapping even asymptotically de Sitter  space of even dimension $n$. Then 
	the Schwartz kernel $K_{\alpha,k}(.,.) $ of $ R_F^{(\alpha,k)}(\pm i\epsilon) $ exists as a family of distributions near the diagonal depending holomorphically in $\alpha$ on the half-plane $\Re\alpha>2-k$.
	Its restriction on the diagonal
	$K_{\alpha,k}(x,x)$ exists and is holomorphic for $\Re\alpha>\frac{n}{2} $, and it extends as a meromorphic 
	function of $\alpha$ with simple poles along the arithmetic progression
	$\{\frac{n}{2}, \frac{n}{2}-1,\dots,1\}$. Furthermore,
	$$
	\bea 
	\lim_{\epsilon\rightarrow 0^+} {\rm res}_{\alpha=\frac{n}{2}-\varm}R_F^{(\alpha,k)}(\pm i\epsilon)(x,x)= 
	\mp\frac{ i \,u_\varm(x,x)}{2^n\pi^{\frac{n}{2}}(\frac{n}{2}-\varm-1)!}.
	\eea $$
\end{theorem}

\begin{proof}
Since $R_F^{(\alpha,k)}(\mu+i\epsilon,.)=H_N^{(\alpha)}(\mu+i\epsilon,.)+R_N(i\epsilon,\alpha,k)$ near the diagonal and $R_N(i\epsilon,\alpha,k)\in C^l(\pazocal U)$ we have that we can take its restriction to the diagonal, and the same is true for the Hadamard parametrix so it must be true for $R_F^{(\alpha,k)}(\mu+i\epsilon,.)$. Take $p=0$ and we get the meromorphy statement. Since $R_N(i\epsilon,\alpha,k)(x,x)$ is holomorphic for $\Re\alpha>2-k$ it has no poles therefore the residue statement is only about the poles of the complex powers of the Hadamard parametrix which are well understood and described in this section.
\end{proof}

The particular case $m=1$ of the theorem gives explicitly the spectral action:

\begin{theorem}\label{mainmainthm2}
	Let $(M,g)$ be a non-trapping even asymptotically de Sitter  space of even dimension $n$. Then 
	the Schwartz kernel $K_s(.,.) $ of $ R_F^{(\alpha,k)}(\pm i\epsilon) $ exists as a family of distributions near the diagonal depending holomorphically in $\alpha$ on the half-plane $\Re\alpha>2-k$.
	Its restriction on the diagonal
	$K_\alpha(x,x)$ exists and is holomorphic for $\Re\alpha>\frac{n}{2} $, and it extends as a meromorphic 
	function of $\alpha$ with simple poles along the arithmetic progression
	$\{\frac{n}{2}, \frac{n}{2}-1,\dots,1\}$. Furthermore,
	$$
	\lim_{\epsilon\rightarrow 0^+} {\rm res}_{\alpha=\frac{n}{2}-1}R_F^{(\alpha,k)}(\pm i\epsilon)(x,x)= 
	\pm\frac{ i \,R_g(x)}{6(4\pi)^{\frac{n}{2}}(\frac{n}{2}-2)!}.
		 $$
	\end{theorem}

\appendix

\section{ Lemmas}

We recall the following lemma and its proof from Appendix D of \cite{DanWro}.
\begin{lemma}\label{lem:holdersobolev}
	Let $s\in \mathbb{R}$. Then
	$u\in \pazocal{C}^s_{\loc}(\mathbb{R}^n)$ iff for every test function $\chi\in C^\infty_\c(\mathbb{R}^n)$
	$$ \vert \widehat{u\chi}(\xi) \vert\leqslant C(1+\vert \xi\vert)^{-s-n} .$$
	{As a consequence, we have the continuous injection $\pazocal{C}^{s}_{\loc}(\mathbb{R}^n) \hookrightarrow H_{\loc}^{s+\frac{n}{2}-\epsilon} (\mathbb{R}^n)$ for all $\epsilon>0$.}
\end{lemma}

\begin{proof}
	If $u\in \pazocal{C}^s(\mathbb{R}^n)$ with $s>0, k< s <k+1$ then it means for any $x$, there exists a polynomial $P$ of degree $k$, which is nothing but the Taylor polynomial of $u$ at $x$, s.t.~{for all test function $\varphi\in C^\infty_{\rm c}(\mathbb{R}^n)$ (one could also take $\varphi$ in the Schwartz class):} 
	$$\bigg| \int_{\mathbb{R}^n}(u-P)(\lambda(y-x)+x)\varphi(y)d^ny\bigg|\leqslant C\lambda^s\Vert \varphi\Vert_{L^\infty} .$$
	Now let $u\in \pazocal{C}^s_{\loc}(\mathbb{R}^n)$, hence we may multiply
	$u$ with some cut-off $\chi\in C^\infty_\c(\mathbb{R}^n)$ so that
	$u\chi\in \pazocal{C}^s$. 
	In particular choosing the function on the r.h.s.~as $ e^{ix.\xi}$ yields 
	$$ \sup_{0<\lambda\leqslant 1} \lambda^{-s} \sup_{1\leqslant\vert \xi\vert \leqslant 2} \module{\left\langle   (u\chi-P)(\lambda.), e^{i\xi.x}\right\rangle}\leqslant  \Vert u\chi\Vert_{\pazocal{C}^s} \sup_{1\leqslant\vert \xi\vert \leqslant 2} \Vert e^{i\left\langle\xi,.\right\rangle } \Vert_{L^\infty}= \Vert u\chi\Vert_{\pazocal{C}^s}. $$
	Therefore, {using the fact} that $\left\langle   (u\chi-P)(\lambda.),e^{i\xi.x}\right\rangle=\left\langle   u\chi(\lambda.), e^{i\xi.x}\right\rangle  $ since the Fourier transform restricted to $\vert\xi\vert\geqslant 1$ does not see the polynomial, and 
	$\left\langle   u\chi(\lambda.), e^{i\xi.x}\right\rangle=\lambda^{-n}\widehat{u\chi}(\frac{\xi}{\lambda}) $,
	we get
	$$ \sup_{0<\lambda\leqslant 1} \lambda^{-s-n} \sup_{1\leqslant\vert \xi\vert \leqslant 2} \vert\widehat{u\chi}(\xi/\lambda)\vert\leqslant \Vert u\chi\Vert_{\pazocal{C}^s} .$$
	Hence for $\vert \xi\vert\geqslant 1$, we get
	$$\vert\widehat{u\chi}(\xi)\vert= \vert\widehat{u\chi}(\xi\vert\xi\vert/\vert \xi\vert)\vert\leqslant \Vert u\chi\Vert_{\pazocal{C}^s}\vert \xi\vert^{-s-n} $$
	and finally this means 
	that: 
	$$ \vert \widehat{u\chi}(\xi) \vert\leqslant C(1+\vert \xi\vert)^{-s-n} .$$
	
	Conversely, if we have the Fourier decay 
	$ \vert\widehat{u\chi}(\xi)\vert\leqslant C(1+\vert\xi \vert)^{-r} $ for $r\in \mathbb{R}_{\geqslant 0}$, then
	the Littlewood--Paley blocks are bounded by: 
	$$ \bea 
	\Vert \psi(2^{-j}\sqrt{-\Delta}) (u\chi) \Vert_{L^\infty} &=\Vert \pazocal{F}^{-1}\left( \psi(2^{-j}\vert \xi\vert) \widehat{u\chi}(\xi)\right) \Vert_{L^\infty}
	\leqslant \int_{\mathbb{R}^n} \vert  \psi(2^{-j}\vert \xi\vert) \widehat{u\chi}(\xi) \vert d^n\xi
	\\
	&\leqslant 
	2^{jn}\int_{\mathbb{R}^n} \vert  \psi(\vert \xi\vert) \widehat{u\chi}(2^{-j}\xi) \vert d^n\xi\leqslant 
	C2^{jn} \int_{\mathbb{R}^n} \psi(\vert \xi\vert)(1+2^{j}\vert \xi\vert)^{-r}d^n\xi\\
	&\leqslant C2^{j(n-r)}  \int_{\mathbb{R}^n} \psi(\vert \xi\vert)(2^{-j}+\vert \xi\vert)^{-r}d^n\xi\lesssim 2^{j(n-r)}.
	\eea $$ 
	This means that $u\in \pazocal{C}^{ r-n}_{\loc}(\rr^n)$.
\end{proof}

\begin{lemma}\label{sobsch}
	Let $A(z):\pazocal O_{H^*\to H^{*+s}}(h(z))$ for $s\in\rr$ and let $K(z,.)$ be the Schwartz kernel of $A$. We have that $K(z,.)\in H^s(M\times M)$ with the norm being $\pazocal O(h(z))$.
\end{lemma}
\begin{proof}
	We have that $A(z):\pazocal O_{H^{s-k}\to H^{-k}}(h(z))$ for all $k$ so let $B_1, B_2$ be elliptic pseudo-differential operators of respective order $k$ and $k-s$ and $u,v\in L^2$. We have that
	$
	\bra v,B_2AB_1u\ket <+\infty
	$
	but
	\beq
	\bea
	\bra v,B_2AB_1u\ket &=\bra B_2^* v,AB_1u\ket= \int B_2^* \bar v(x)K(x,y)B_1u(y)dxdy\\
	&=  \int  \bar v(x)B_{2,x}B_{1,y}K(x,y)u(y)dxdy=\bra v\otimes \bar u,B_{2,x}B_{1,y}K  \ket
	\eea
	\eeq
	and by Cauchy-Schwartz $\vert \bra v\otimes \bar u,B_{2,x}B_{1,y}K  \ket\vert \leq \norm{u}\norm{v}\norm{A}$ so $\norm{K(z,.)}_{H^s}\leq \norm{A(z)}$.
\end{proof}

\bibliographystyle{abbrv}

\bibliography{Paper_spectral_action-short}

\end{document}